\documentclass[leqno,12pt]{article} %leqnoで数式番号を左側へ配置
\usepackage{amsmath, amssymb}
\usepackage{amsthm} %AMSのTheorem環境を導入
\theoremstyle{plain} % イタリック体
\newtheorem{theorem}{\indent\sc Theorem}[section] % 見出しはスモールキャップ
\newtheorem{lemma}[theorem]{\indent\sc Lemma}
\newtheorem{corollary}[theorem]{\indent\sc Corollary}
\newtheorem{proposition}[theorem]{\indent\sc Proposition}

\theoremstyle{definition} % ローマン体に変更

\newtheorem{remark}[theorem]{\indent\sc Remark}
\newtheorem{example}[theorem]{\indent\sc Example}

\newcommand{\ellOne}{%}
\begin{picture}(0,0)
\setlength{\unitlength}{1.08358pt}
{\thicklines%}
\put(6.05,0){\circle{12}}
\put(6.05,0){\circle{16}}
}%\thicklines
\qbezier(13.5,0)(18,0)(22,0)
\put(6.15,0){\makebox(0,0){$1$}}
\end{picture}}
\newcommand{\hidariashiEllOne}{%}
\begin{picture}(0,0)
\setlength{\unitlength}{1.08358pt}
{\thicklines%}
\put(7.55,0){\circle{12}}
\put(7.55,0){\circle{16}}
}%\thicklines
\put(7.65,0){\makebox(0,0){$1$}}
\end{picture}}
\newcommand{\migite}{%}
\begin{picture}(0,0)
\setlength{\unitlength}{1.08358pt}
{\thicklines%}
\put(6.05,0){\circle{12}}
}%\thicklines
\qbezier(12,0)(17,0)(22,0)
\end{picture}}
\newcommand{\hidariashi}{%}
\begin{picture}(0,0)
\setlength{\unitlength}{1.08358pt}
{\thicklines%}
\put(6.05,0){\circle{12}}
}%\thicklines
\end{picture}}
\newcommand{\hidariaOne}{%}
\begin{picture}(0,0)
\setlength{\unitlength}{1.08358pt}
{\thicklines%}
\put(6.05,0){\circle{12}}
}%\thicklines
\put(6,0){\makebox(0,0){$1$}}
\end{picture}}
\newcommand{\migitesan}{%}
\begin{picture}(0,0)
\setlength{\unitlength}{1.08358pt}
{\thicklines%}
\put(6.05,0){\circle{12}}
}%\thicklines
\qbezier(12,0)(17,0)(22,0)
\put(6,0){\makebox(0,0){$3$}}
\end{picture}}
\newcommand{\FLEight}{%}
\begin{picture}(141,50)
\setlength{\unitlength}{1.08358pt}
{\thicklines%}
\put(96,37.5){\circle{12}}
}%\thicklines
\qbezier(96,21.5)(96,26.5)(96,31.3)
\put(0,15){\migite}
\put(22.5,15){\migite}
\put(45,15){\migite}
\put(67.5,15){\migite}
\put(90,15){\migite}
\put(112.5,15){\migite}
\put(135,15){\hidariashi}
\end{picture}}
\newcommand{\Eight}{%}
\begin{picture}(141,50)
\setlength{\unitlength}{1.08358pt}
{\thicklines%}
\put(96,37.5){\circle{12}}
}%\thicklines
\qbezier(96,21.5)(96,26.5)(96,31.3)
\put(0,15){\migite}
\put(22.5,15){\migite}
\put(45,15){\migite}
\put(67.5,15){\migite}
\put(90,15){\migite}
\put(112.5,15){\migite}
\put(135,15){\hidariashi}
\end{picture}}
\newcommand{\EightM}{%}
\begin{picture}(141,50)
\setlength{\unitlength}{1.08358pt}
{\thicklines%}
\put(6.05,15){\circle{12}}
}%\thicklines
\qbezier(12,15)(17,15)(22,15)
\put(6,15){\makebox(0,0){$4$}}
{\thicklines%}
\put(73.5,37.5){\circle{12}}
}%\thicklines
\qbezier(73.5,21.5)(73.5,26.5)(73.5,31.3)
\put(22.5,15){\migite}
\put(45,15){\migite}
\put(67.5,15){\migite}
\put(90,15){\migite}
\put(112.5,15){\migite}
\put(135,15){\migite}
\put(157.5,15){\hidariashi}
\end{picture}}
\newcommand{\V}{%}
\begin{picture}(68,50)
\setlength{\unitlength}{1.08358pt}
{\thicklines%}
\put(28.5,37.5){\circle{12}}
}%\thicklines
\qbezier(28.5,21.5)(28.5,26.5)(28.5,31.3)
\put(0,15){\migite}
\put(22.5,15){\migite}
\end{picture}}
\makeatletter
\def\address#1#2{\begingroup
\noindent\parbox[t]{7.8cm}{%
\small{\scshape\ignorespaces#1}\par\vskip1ex
\noindent\small{\itshape E-mail address}%
\/: #2\par\vskip4ex}\hfill%
\endgroup}%
\makeatother
\title{\uppercase{Pencils of genus two curves on rational surfaces}} %論文タイトル大文字
\author{
\textsc{Shinya Kitagawa} %著者名
}
\date{} % 日付けは記入しない
\makeatletter
\long\def\@makecaption#1#2{%
\vskip\abovecaptionskip%}
\sbox\@tempboxa{#1. #2}
\ifdim \wd\@tempboxa >\hsize%
#1 #2\par
\else
\global \@minipagefalse
\hb@xt@\hsize{\hfil\box\@tempboxa\hfil}%
\fi
\vskip\belowcaptionskip}
\makeatother
\begin{document}
\setlength{\baselineskip}{6.5mm}%%%%%%%%%%%%%%%%%%%%%%%%%%%%%%%%%%%%%行間調整

\maketitle

%%%%%%%%%%%%%%%%%%% 脚注 %%%%%%%%%%%%%%%%%%%%%%%%%%%%%%
\footnote{ %2000MSC
}
\footnote{ %keywords
}
%\footnote{ %Thanks
%$^{*}$Partly supported by the Grant-in-Aid for Scientific Research (A),
%Japan Society for the Promotion of Science. 
%%科研費補助　基盤研究(A)の場合
%}

\begin{abstract}
We consider relatively minimal fibrations of curves of genus two on rational surfaces 
whose Picard numbers are not maximal. 
By birational morphisms, such fibred surfaces are interpreted as pencils of plane curves. 
We show that only four are canonical, 
among a variety of possible models. 
For each canonical pencil, we give an example with trivial Mordell-Weil group. 
\end{abstract}
%
%------------------- 序 -------------------
%
\section{Introduction}
The theory of the Mordell-Weil lattices are sufficiently developed by 
Oguiso and Shioda in \cite{OS} for minimal elliptic rational surfaces. 
In their work, the even unimodular root lattice $E_8$ of rank eight played very important role 
as the predominant frame. 
For example, it was shown that the Mordell-Weil group is trivial if and only if 
there exists a singular fibre of type $\mathrm{II}^*$ in the sense of Kodaira \cite{Kodai63} 
whose dual graph contains $E_8$ as a subgraph. 
The lattice $E_8$ also appears in another application by Shioda \cite{Shioda91} to describe 
a hierarchy of deformations of rational double points.

In this paper, we consider fibred rational surfaces of genus two over $\mathbb{C}$ 
in order to look for right candidates for the ``frame lattices'' in this case. 
Here, a fibred rational surface of genus two means a smooth projective rational surface 
$X$ together with a relatively minimal fibration $f:X\to \mathbb{P}^1$ 
whose general fibre $F$ is a smooth projective curve of genus two. 
It is shown in Proposition~\ref{section} that $f$ has a section. 
Hence we can always associate to $f$ the Mordell-Weil group (or lattice) 
by applying the general machinery of Shioda~\cite{Shioda99}. 
We regard a section of $f$ as a curve on $X$, and call it a $(-1)$-{\it section} of $f$ 
if the self-intersection number is equal to $(-1)$. 
To clarify the structure of the Mordell-Weil lattice, 
we choose a ruling on $X$ and its relatively minimal model $\Sigma_d$ 
%$\Sigma_d:=\mathbb{P}(\mathcal{O}_{\mathbb{P}^1}\oplus\mathcal{O}_{\mathbb{P}^1}(d))\to\mathbb{P}^1$ 
carefully so that we get a natural $\mathbb{Z}$-basis of $\mathrm{NS}(X)$ 
(the N\'eron-Severi group) which gives us a simple presentation of $F$. 
This is done by choosing a birational morphism $X\to\Sigma_d$ 
which contracts step by step a $(-1)$-curve whose intersection number with $F$ 
is the smallest among all $(-1)$-curves.

Let $\rho(X)$ denote the Picard number of $X$. 
Then it can be shown that $\rho(X)\leq14$. 
When $\rho(X)=14$ and the Mordell-Weil group of $f$ is non-trivial, as an above procedure, 
it is not so hard to see that there are 
twelve $(-1)$-curves each of which meets $F$ at one point and one $(-1)$-curve meeting $F$ at two points 
such that, by contracting them all, we get a pencil of plane quartic curves whose base points 
are resolved to get $|F|$. 
%(cf.~\cite{SS}, \cite{Viet00}). 
If the group is trivial, then $X$ has a unique ruling and we can only take 
$\Sigma_2$ or $\Sigma_3$ as its minimal model, 
while there actually exists some $f$'s not admitting such a model 
when the Mordell-Weil group is non-trivial (cf.~\cite{Kitagawa5}). 
In other words, $\mathbb{P}^2$ with images become such models of more $f$'s 
than $\Sigma_2$ and $\Sigma_3$.

When $\rho(X) \leq 13$, by choosing birational morphisms like the one above, 
$\mathbb{P}^2$ with images of $|F|$'s become such models for all $f$'s, 
while $\Sigma_d$ with any images can not be models of some $f$'s for $d\not=1$ 
(cf.\ Proposition~\ref{Minimini} and Theorem~\ref{PlanePlan}). 
Furthermore, the main theorem describes the such models on $\mathbb{P}^2$ as follows:

%
%=%=%=%=%=%=%=%=%=%=%=%=%=%=%=%=%=%=%=%=%=%=%=%=%=%=%=%=%=%=%=%=%=%=%=%=%=%=%=%=%=%=%
%
\begin{theorem}[cf.\ {\upshape Theorem~\ref{MT}}]\label{IntMT}
Let $f:X\to\mathbb{P}^1$ be a fibred rational surface of genus two. 
Assume that $\rho(X)\not=14$. 
Then there exists a birational morphism $\upsilon_0:X\to\mathbb{P}^2$ 
such that $\upsilon_0(F)$ is one of the following$:$ 
\begin{itemize}
\item[$(\mathrm{A})$]
In the case where $\rho(X)=13$ and $f$ has a $(-1)$-section, 
$\deg\upsilon_0(F)=6$ and singularities of $\upsilon_0(F)$ are eight double points. 
\item[$(\mathrm{B}1)$]
In the case where $\rho(X)=12$ and $f$ has no $(-1)$-section, 
$\deg\upsilon_0(F)=7$ and singularities of $\upsilon_0(F)$ are 
one triple point and ten double points. 
\item[$(\mathrm{B}2)$]
In the case where $\rho(X)=12$ and $f$ has a $(-1)$-section, 
$\deg\upsilon_0(F)=9$ and singularities of $\upsilon_0(F)$ are 
eight triple points and two double points. 
\item[$(\mathrm{C})$]
In the case where $\rho(X)=11$ and $f$ has no $(-1)$-section,
$\deg\upsilon_0(F)=13$ and singularities of $\upsilon_0(F)$ are 
one quintuple point and nine quadruple points. 
\end{itemize}

Furthermore, $\deg\upsilon_0'(F)\geq\deg\upsilon_0(F)$, 
%for any birational morphism $\upsilon_0':X\to(\mathbb{P}^2)'$. 
for any birational morphism $\upsilon_0':X\to\mathbb{P}^2$. 
If the equality sign holds, 
then the types of singularities of $\upsilon_0'(F)$ are the same as $\upsilon_0(F)$'s. 
\end{theorem}
%
%=%=%=%=%=%=%=%=%=%=%=%=%=%=%=%=%=%=%=%=%=%=%=%=%=%=%=%=%=%=%=%=%=%=%=%=%=%=%=%=%=%=%
%

%Indeed, any birational morphism $\upsilon_0:X\to\mathbb{P}^2$ in Theorem~\ref{IntMT} 
%has the minimality for $F$. 
%Furthermore, the last two statements in Theorem~\ref{IntMT} imply that four pencils in 
The last two statements in Theorem~\ref{IntMT} imply that four pencils in 
$(\mathrm{A})$, $(\mathrm{B}1)$, $(\mathrm{B}2)$ and $(\mathrm{C})$ are canonical. 
Furthermore, we expect the followings to be the frame lattices.

\begin{theorem}[cf.\ {\upshape \cite{Kitagawa6}}]
For fibred surfaces in cases $(\mathrm{A})$, $(\mathrm{B}1)$, $(\mathrm{B}2)$ and 
$(\mathrm{C})$ respectively, Mordell-Weil lattices of the maximal rank $2(\rho(X)-8)$ 
are isomorphic to a unimodular integral lattice whose extended Dynkin diagram is given by 
{\upshape Figures}~$\ref{fig:FL(A)}$, $\ref{fig:FL(B1)}$, $\ref{fig:FL(B2)}$ 
and $\ref{fig:FL(C)}$. 
%%%%%%%%%%%%%%%%%%%%%%%%%%%%%%%%%%%%%%%
\begin{figure}[hbtp]
\begin{center}
\begin{picture}(214.5,35)
\setlength{\unitlength}{1.08358pt}
%
%\qbezier(0,10)(0,0)(0,-10)
%\qbezier(107.25,10)(107.25,0)(107.25,-10)
%\qbezier(84.75,-6.5)(107.25,-6.5)(129.75,-6.5)
%\qbezier(214.5,10)(214.5,0)(214.5,-10)
%
\put(0,0){\migite}
\put(22.5,0){\hidariaOne}
\put(67.5,-15){\FLEight}
\end{picture}
\caption{}\label{fig:FL(A)}
\end{center}
\end{figure}
%%%%%%%%%%%%%%%%%%%%%%%%%%%%%%%%%%%%%%%
\begin{figure}[hbtp]
\begin{center}
\begin{picture}(192,18)
\setlength{\unitlength}{1.08358pt}
%
%\qbezier(0,10)(0,0)(0,-10)
%\qbezier(96,10)(96,0)(96,-10)
%\qbezier(50,-6.5)(96,-6.5)(142,-6.5)
%\qbezier(192,10)(192,0)(192,-10)
%
\put(0,0){\migite}
\put(22.5,0){\migite}
\put(45,0){\migite}
\put(67.5,0){\hidariaOne}
\put(112.5,0){\migite}
\put(135,0){\migite}
\put(157.5,0){\migite}
\put(180,0){\hidariaOne}
\end{picture}
\caption{}\label{fig:FL(B1)}
\end{center}
\end{figure}
%%%%%%%%%%%%%%%%%%%%%%%%%%%%%%%%%%%%%%%
\begin{figure}[hbtp]
\begin{center}
\begin{picture}(147,35)
\setlength{\unitlength}{1.08358pt}
%
%\qbezier(0,10)(0,0)(0,-10)
%\qbezier(73.5,10)(73.5,0)(73.5,-10)
%\qbezier(48.5,-6.5)(73.5,-6.5)(98.5,-6.5)
%\qbezier(147,10)(147,0)(147,-10)
%
\put(0,-15){\FLEight}
\end{picture}
\caption{}\label{fig:FL(B2)}
\end{center}
\end{figure}
%%%%%%%%%%%%%%%%%%%%%%%%%%%%%%%%%%%%%%%
\begin{figure}[hbtp]
\begin{center}
\begin{picture}(190,18)
\setlength{\unitlength}{1.08358pt}
%
%\qbezier(0,10)(0,0)(0,-10)
%\qbezier(85,10)(85,0)(85,-10)
%\qbezier(70,-6.5)(85,-6.5)(100,-6.5)
%
\put(0,0){\migite}
\put(22.5,0){\hidariaOne}
\put(67.5,0){\migite}
\put(90,0){\hidariaOne}
\put(135,0){\migite}
\put(157.5,0){\hidariaOne}
\end{picture}
\caption{}\label{fig:FL(C)}
\end{center}
\end{figure}
%%%%%%%%%%%%%%%%%%%%%%%%%%%%%%%%%%%%%%%
Here the numbers in the circles denote the self-pairings of elements except for roots, 
and a line between two circles shows 
that the pairing of the corresponding two elements is equal to $(-1)$.

Furthermore, any other Mordell-Weil lattice of a fibred surface as above is isomorphic to 
the dual lattice of at most a sublattice of the maximal one. 
\end{theorem}

In particular, the maximal Mordell-Weil lattice is $E_8$ in the case $(\mathrm{B}2)$. 
On the other hand, that in the case $(\mathrm{B}1)$, or as in Figure~$\ref{fig:FL(B1)}$ 
is an odd lattice. 
%As one may learn from $(\mathrm{B}1)$ and $(\mathrm{B}2)$, 
Therefore, 
the existence of a $(-1)$-section affects the structure of $f$ essentially. 
%(see also Remark~\ref{Rmk:SINsec}). 
In the last of this paper, for each case in Theorem~\ref{IntMT}, 
we describe an example which is extremal in the sense that Mordell-Weil group of $f$ is trivial. 
The sum of all dual graphs of reducible fibres of $f$ in the example, 
or 
as in Figures~\ref{fig:H1V3}, \ref{fig:H2VV}, \ref{fig:H2II1} and \ref{fig:H3VVV} respectively 
contains the extended Dynkin diagram 
as in Figures~$\ref{fig:FL(A)}$, $\ref{fig:FL(B1)}$, $\ref{fig:FL(B2)}$ and $\ref{fig:FL(C)}$.
%We expect these lattices to play a role like $E_8$-frame in the elliptic case. 
These are the same as in the elliptic case, and we are interested in 
singularities obtained from the reducible fibres by contracting irreducible components.

A sketch of the proof of Theorem~\ref{IntMT} is as follows: 
For a fibred rational surface $f:X\to\mathbb{P}^1$ of genus two with $\rho(X)\leq13$, 
we consider a \#-minimal model of the reduction of $(X,F)$. 
The method as in \cite[\S2]{Kitagawa3} leaves 
five numerical possibilities for the \#-minimal models. 
However, we can exclude one of them in the course of the study of the branch divisor, 
associated with the relative canonical map classified by Horikawa \cite{Hori77}. 
Now, a procedure taking the \#-minimal model 
gives a birational morphism $X\to\mathbb{P}^2$ naturally. 
%determines a birational morhism 
%$\upsilon_0:X\to\mathbb{P}^2$ in Theorem~\ref{IntMT}. 
We pay a special attention to $\rho(X)$ and the existence of a $(-1)$-section of $f$. 
Although a \#-minimal model is not unique in general, 
%(cf.~Remark~\ref{ABM}), 
the four cases in Theorem~\ref{IntMT} correspond in a one-to-one manner to the four types of \#-minimal models. 
By comparing base-point-free pencils of rational curves on $X$ with a minimal one, 
we have the last two statements. 
%
%%%%%%%%%%%%%%%%%%%%%%%%%%%%%%%%%%%
%

\medskip

A{\sc cknowledgment}. 
The author would like to express his heartfelt gratitude to 
Professor Kazuhiro Konno for his valuable advice, guidance and encouragement. 
Thanks are also due to Professors Tetsuji Shioda, Tadashi Ashikaga 
and Doctors Atsushi Ikeda, Mizuho Ishizaka, Takeshi Harui for interesting discussions. 

%In particular, he told Proposition~\ref{section} to the author and
%independently gave another proof of Proposition~\ref{Prop:UBAdjDeg} by using cohomologies. 
%Shioda Tetsuji, 
%Ashikaga, Kuwata, Ishizaka, Harui, Kawaguchi, Yabu
%, 
%and Doctors , 
%
%and 
%for interesting discussions, useful advises, and encouragements. 

%Thanks are also due to Professors Hisashi Usui and Viet Nguyen Khac for interesting discussions. 
%Arakawa Tatsuya, 
%Kojima Hideo, 
%Oguiso, Keiji, 
%Ohbuchi Akira, 
%Takahashi Takeshi, 
%Usui Hisashi, 
%Yoshihara Hisao 
%Ishida Hirotaka, 
%Murakami Masaaki 
%
%
%
%
%
%------------------- 有理曲面のスロープ不等式 -------------------
%
%
%
%
%
\section{Preliminaries}\label{Sec:RA}
Let $X$ be a smooth projective rational surface defined over $\mathbb{C}$ and 
$f:X\to \mathbb{P}^1$ a relatively minimal fibration 
whose general fibre $F$ is a smooth projective curve of genus $g\geq2$. 
Then $K_X+F$ is nef and it follows that 
the self-intersection number of a section of $f$ is negative. 
Furthermore, from Noether's formula and the genus formula, the Picard number $\rho(X)$ of $X$ is as follows: 
\begin{align}
\rho(X) =4g+6 - (K_X+F)^2\leq 4g+6. \label{Eq:PicNum}
\end{align}
Put $\Sigma_d=\mathbb{P}(\mathcal{O}_{\mathbb{P}^1}\oplus\mathcal{O}_{\mathbb{P}^1}(d))\to\mathbb{P}^1$. 
Let $\Delta_0$ be a section of $\Sigma_d$ with $\Delta_0^2=-d$ and $\Gamma$ a fibre. 
When $\rho(X)=4g+6$, that is $(K_X+F)^2=0$, we obtain $f:X\to\mathbb{P}^1$ 
from a suitable subpencil $\Lambda_f$ of $|2\Delta_0+(d+g+1)\Gamma|$ with $0\leq d \leq g+1$ 
by blowing $\Sigma_d$ up at the $(4g+4)$ base points 
(cf.\ \cite[Theorem~4.1]{SS} and \cite[Theorem~2.4]{Viet00}, see also \cite{Kitagawa5}). 
%(cf.\ \cite[Theorem~4.1]{SS}, \cite[Theorem~2.2]{Viet00} and \cite[Lemma~2.8]{Kitagawa3}). 
In particular, $f$ always has a $(-1)$-{\it section} $E$, i.e., a $(-1)$-curve with $F.E=1$.

Assume that $(K_X+F)^2>0$. 
%We usually denote by $F$ a general fibre of $f$, but it sometimes represents an arbitrary fibre. 
We briefly review basic notation and results of such fibred surfaces $f:X\to \mathbb{P}^1$ 
%while considering the adjoint bundle $K_X+F$ according to \cite{KK} and \cite{Kitagawa3}. 
according to \cite{KK} and \cite{Kitagawa3}.

Suppose that there exists a $(-1)$-curve $E$ with $(K_X+F).E=0$ and 
let $\mu_1:X\to X_1$ be its contraction. 
Since $F.E=1$, $F_1:=(\mu_1)_*F$ is smooth on $X_1$. 
Furthermore, we have $\mu_1^*(K_{X_1}+F_1)=K_X+F$. 
If there exists a $(-1)$-curve $E_1$ with $(K_{X_1}+F_1).E_1=0$, then, 
by contracting it, we get the pair $(X_2,F_2)$ with $F_2$ smooth 
and $K_{X_2}+F_2$ pulls back to $K_X+F$. 
We can continue the procedure until we arrive at a pair $(X_n,F_n)$ 
such that we cannot find a $(-1)$-curve $E_n$ with $(K_{X_n}+F_n).E_n=0$. 
We put $Y:=X_n$ and $G:=F_n$. 
If $\mu:X\to Y$ denotes the natural map, then $\mu^*(K_Y+G)=K_X+F$ and 
$G=\mu_*F$ is a smooth curve isomorphic to $F$. 
The original fibration $f:X\to \mathbb{P}^1$ corresponds to a pencil 
$\Lambda_f\subset |G|$ with at most simple 
(but not necessarily transversal) base points. 
Remark that 
\begin{align}
G^2=\#\mathrm{Bs}\Lambda_f\geq0, \ \ \ {K_X}^2={K_Y}^2-G^2.\label{poG}
\end{align}
Since $K_X+F$ is nef and big, 
$Y$ is the minimal resolution of singularities of the surface ${\rm Proj}(R(X,K_X+F))$, 
which has at most rational double points from \cite[Lemma~1.2]{KK}, 
where $R(X,K_X+F)=\bigoplus_{n\geq 0} H^0(X,n(K_X+F))$. 
Therefore, such a model is uniquely determined. 
We call the pair $(Y,G)$ the {\it reduction} of $(X,F)$.

If $Y=\mathbb{P}^2$, then $G$ is a smooth plane curve of degree $b$ with $b\geq4$. 
In this case, we have $g=(b-1)(b-2)/2$ and 
\begin{align}
\label{eq:Plane}
(K_X+F)^2=(K_Y+G)^2=(b-3)^2=4g+5-b^2\leq 4g-11. 
\end{align}
In the case of $Y=\Sigma_d$, from \cite[Lemma~2.5]{Kitagawa3}, we have 
\begin{align}
\label{eq:Hd}
(K_X+F)^2=2c(g-c-1)/(c+1)<2g-2 
\end{align}
with the Clifford index $c$ of $F$, which is a non-negative integer.

Assume that $Y$ is neither $\mathbb{P}^2$ nor $\Sigma_d$. 
Then we can find at least one base-point-free pencil of rational curves on $Y$. 
We choose among them a pencil $|\Gamma_Y|$ of rational curves with ${\Gamma_Y}^2=0$ 
in such a way that $a:=(K_Y+G).\Gamma_Y$ is minimal. 
%
%We call $a$ the {\it minimal ruling degree of} $(Y,G)$. 
%
Note that $a>0$ from \cite[Lemma~1.2]{KK}. 
We have $G.\Gamma_Y=a+2$ since $K_Y.\Gamma_Y=-2$. 
Let $\psi:Y\to \mathbb{P}^1$ be the morphism defined by $|\Gamma_Y|$. 
We take a relatively minimal model of $Y$ with respect to $\psi$ 
and consider the image of $G$. 
Then we perform a succession of elementary transformations 
(\cite{Hartshorne}) at singular points of the image curve 
to arrive at a particular relatively minimal model $(Y^\#,G^\#)$, 
called a \#-minimal model in \cite{Iitaka}, enjoying several nice properties which we collect below. 
The natural map $\nu:(Y,G)\to (Y^\#,G^\#)$ is a minimal succession of 
blowing-ups which resolves the singular points of $G^\#$. 
We assume that $Y^\#\simeq \Sigma_d$ and $G^\#\sim(a+2)\Delta_0+b\Gamma$, 
where the symbol $\sim$ means the linear equivalence of divisors. 
%$(a+2)\Delta_0+b\Gamma$, where $\Delta_0$ and $\Gamma$ respectively 
%denote a section with $\Delta_0^2=-d$ and a fibre of $\Sigma_d$. 
Let $p_i$, $1\leq i\leq N$, be the singular points of $G^\#$ including infinitely near ones, 
and let $m_i$ be the multiplicity of $G^\#$ at $p_i$. 
Assume for simplicity that $m_1\geq m_2\geq \cdots \geq m_N\geq 2$. 
%unless $G^\#$ is smooth. 
Since $|\Gamma_Y|$ is chosen so that $(K_Y+G).\Gamma_Y$ is minimal, we can assume that 
the following are satisfied (see \cite{Hartshorne} and \cite{Iitaka}): 

\medskip

(\#1) $b\geq (a+2)d$ when $d>0$, and $b\geq a+2$ when $d=0$, 

\medskip

%(\#2) $b\geq a+2+m_1$ when $d=1$,
%(\#2) $b\geq \max\{a+4,\ a+2+m_1\}$ when $d=1$,
%
%\medskip

(\#2) $m_1\leq (a+2)/2$ and $m_1\leq \min\{(a+2)/2,\ b-(a+2)\}$ when $d=1$.

\medskip

\noindent
For another birational morphism $\nu':(Y, G)\to(\Sigma_{d'}, G')$ such that 
$a':=(K_Y+G).(\nu')^*\Gamma'$ is not minimal, 
we also call $(\Sigma_{d'}, G')$ a \#-minimal model of $(Y, G)$ if $(\Sigma_{d'}, G')$ satisfies 
(\#1) and (\#2) with $m'_1\geq m'_2\geq\cdots\geq m'_N\geq2$. 
%
%
%
%特に以降では $a$ が極小であるとは限らない. 
%
%
%
In fact, $a\geq 2$ holds from $(\#2)$. 
We say that $(Y^\#,G^\#)$ is of {\it general type} if $2b-(a+2)d\geq 2(a+2)$.
Otherwise, i.e., when $d=1$ and $2b< 3(a+2)$, 
the pair is called of {\it special type}.
If this is the case, by contracting the minimal section, we get 
a model of $G$ which is a plane curve of degree $b$ with a $(b-a-2)$-ple 
point and $N$ other singular points of respective multiplicities $m_i$ $(\leq b-a-2)$. 
We note that a \#-minimal model of special type does not appear when $a=2$.

At first, consider the case where $(Y^\#,G^\#)$ is of general type. 
Set $\breve{b}=b-(d+2)(a+2)/2$, 
where $\breve{b}\geq\max\{0, (d-2)(a+2)/2\}$, 
and where $\breve{b}\in\mathbb{Z}$ or $\breve{b}\in\mathbb{Z}[1/2]$ according as $a$ is even or odd. 
Then 
%we have 
\begin{align*}
G^\#\sim(a+2)\Delta_0+(\breve{b}+(d+2)(a+2)/2)\Gamma.
%\label{Eq:GsharpGen}
\end{align*}
From a standard calculation, we have 
\begin{align}
(K_Y+G)^2&=2a(a+\breve{b})-\sum_{i=1}^{N}(m_i-1)^2,		\label{Eq:AdjDegGenS}	\\
g&=(a+1)(a+1+\breve{b})-\frac{1}{2}\sum_{i=1}^{N}m_i(m_i-1),	\label{Eq:gGenS}	\\
%2g&=2(a+1)(a+1+\breve{b})-\sum_{i=1}^{N}m_i(m_i-1),		\label{Eq:gGenS}	\\
G^2&=2(a+2)(a+2+\breve{b})-\sum_{i=1}^{N}m_i^2.		\label{Eq:GSqGenS}
\end{align}
%When $a=1$ holds, $G^\#$ must be a smooth non-hyperelliptic curve and, 
%hence we have $(Y,G)=(Y^\#,G^\#)$ and $(K_X+F)^2=g-2$ holds. 
%
%
%
Compare $(\ref{Eq:AdjDegGenS})$ with $(\ref{Eq:gGenS})$ and 
consider $2(m_i-1)/m_i=(m_i-1)^2/((1/2)m_i(m_i-1))$, 
which is a monotonically increasing function for $m_i$. 
If we know only $\mathfrak{n}$ values of $m_1,\ldots,m_\mathfrak{n}$ among the multiplicities, 
then, by mimicking \cite[Lemmas~2.6 and 2.8]{Kitagawa3}, 
we have a lower bound of $(K_Y+G)^2$ as follows: 
If $m_{\mathfrak{n}+1}\leq\mathfrak{m}\leq m_\mathfrak{n}$ for an integer $\mathfrak{m}$, then 
\begin{align*}
(K_Y+G)^2\geq &\frac{2(\mathfrak{m}-1)}{\mathfrak{m}}
\left( g-(a+1)(a+1+\breve{b})+\frac{1}{2}\sum_{i=1}^\mathfrak{n}m_i(m_i-1)\right)
\\
&+2a(a+\breve{b})-\sum_{i=1}^\mathfrak{n}(m_i-1)^2.
\end{align*}
When the equality sign holds, 
$m_{\mathfrak{n}+1}=\cdots=m_N=\mathfrak{m}$.
%
%By mimicking \cite[Lemmas~2.6 and 2.8]{Kitagawa3}, we have the following:  
%Furthermore, if we ignore the $\mathfrak{n}$ values, then we have the following: 
Furthermore, we have the following: 
\begin{lemma}\label{frak}
Keep the notation and assumptions as above. 
Let $\mathfrak{n}$ be a non-negative integer and $\mathfrak{m}$ an integer. 
If $a$ is even and $m_{\mathfrak{n}+1}\leq\mathfrak{m}$, then
\begin{align*}
(K_Y+G)^2\geq\frac{2(\mathfrak{m}-1)}{\mathfrak{m}}
\left(g-(a+1)(a+1+\breve{b})+\frac{\mathfrak{n}a(a+2)}{8}\right)
+2a(a+\breve{b})-\frac{\mathfrak{n}a^2}{4}.
\end{align*}
When the equality sign holds, 
$m_\mathfrak{n}=(a+2)/2$ and $m_N=\mathfrak{m}$.
\end{lemma}

By putting $\mathfrak{n}=0$ and $\mathfrak{m}=(a+2)/2$ in Lemma~\ref{frak}, 
we have $(K_Y+G)^2\geq2a(g-1+\breve{b})/(a+2)$ if $a$ is even. 
In \cite[Lemma~2.6]{Kitagawa3}, for another case, we had the following: 
%as follows: 
%If $(Y^\#,G^\#)$ is of general type, $a$ is odd and $a\geq3$, then 
If $(Y^\#,G^\#)$ is of general type and $a$ is odd, then 
\begin{align}
(K_Y+G)^2\geq2g(a-1)/(a+1)+2. \label{eq:AdjDegGen}
\end{align}
When the equality sign holds, $\breve{b}=0$, $N\geq1$ and $m_N=(a+1)/2$.

Next, we consider the case where $(Y^\#,G^\#)$ is of special type. 
Let $m_0=b-(a+2)$. 
Then 
%we have 
\begin{align*}
G^\#\sim(a+2)\Delta_0+(a+2+m_0)\Gamma,\ \ \ 2\leq m_0<(a+2)/2
%\label{Eq:Repre:Sp}
\end{align*}
and
\begin{align}
G^2=(a+2)(a+2+2m_0)-\sum_{i=0}^Nm_i^2\geq(a+2)^2+2m_0(a+2)-Nm_0^2.\label{Eq:GSqSpS}
\end{align}
Furthermore, \cite[Lemma~2.7]{Kitagawa3} showed the following:
If $g<a(a+3)/2$ and $a$ is even, then 
\begin{align}
(K_Y+G)^2\geq2g(a-2)/a+4. \label{eq:AdjDegSpEven}
\end{align}
When the equality sign holds, 
$N>4+4/(a-2)$ and $m_N=a/2$. 
If $g<(a+1)(a+2)/2$ and $a$ is odd, then 
\begin{align}
(K_Y+G)^2\geq2g(a-1)/(a+1)+1. \label{eq:AdjDegSpOdd}
\end{align}
When the equality sign holds, 
$N\geq5$ and $m_N=(a+1)/2$.

In the last of this section, we give an upper bound of $(K_X+F)^2$ as follows:

\begin{proposition}\label{Prop:UBAdjDeg}
Let $X$ be a smooth rational surface and 
$f:X\to\mathbb{P}^1$ a relatively minimal fibration of genus $g\geq2$. 
Then $(K_X+F)^2\leq4g-5$ holds. 
\end{proposition}
%
%=%=%=%=%=%=%=%=%=%=%=%=%=%=%=%=%=%=%=%=%=%=%=%=%=%=%=%=%=%=%=%=%=%=%=%=%=%=%=%=%=%=%
%
\begin{proof}
We shall assume that $(K_X+F)^2>0$. 
Let $(Y,G)$ be the reduction of $(X,F)$. 
From $(\ref{eq:Plane})$ and $(\ref{eq:Hd})$, 
we have that $Y$ is neither $\mathbb{P}^2$ nor $\Sigma_d$. 
Furthermore, we take a \#-minimal model $(Y^\#, G^\#)$ is of $(Y, G)$. 
Remark that ${K_Y}^2={K_{Y^\#}}^2-N=8-N$. 
From $(\ref{poG})$ and the genus formula, we have 
\begin{align}
(K_X+F)^2+G^2+N=4g+4.\label{Eq:RedForm}
\end{align}
Hence, we only have to show that $G^2+N\geq9$. 
Assume that $N\leq8$. 
We suppose that $(Y^\#,G^\#)$ is of general type. 
Then we have $G^2\geq(8-N)(a+2)^2/4\geq8-N$ from $(\#2)$ and $(\ref{Eq:GSqGenS})$. 
If $G^2=8-N$ holds, then $a$ is even and $(\breve{b},m_N,N)=(0,(a+2)/2,8)$. 
However, it implies $g=1$, which is a contradiction. 
If $(Y^\#,G^\#)$ is of special type, then we have 
$G^2\geq3a+4\geq13$ 
from the assumption $N\leq8$, $(\#2)$ and $(\ref{Eq:GSqSpS})$. 
\end{proof}
%
%
%
%
%
%------------------- 種数が２の場合 -------------------
%
%
%
%
%
\section{Canonical pencils of genus two curves}
%\section{Canonical models of pencils of genus two curves}
From now on, we concentrate on relatively minimal fibrations $f:X\to\mathbb{P}^1$ of 
curves of genus two on smooth projective rational surfaces with $(K_X+F)^2>0$. 
We recall $(\ref{Eq:PicNum})$. 
We restate Theorem~\ref{IntMT} as follows:

%
%=%=%=%=%=%=%=%=%=%=%=%=%=%=%=%=%=%=%=%=%=%=%=%=%=%=%=%=%=%=%=%=%=%=%=%=%=%=%=%=%=%=%
%
\begin{theorem}\label{MT}
Let $f:X\to\mathbb{P}^1$ be a fibred rational surface of genus two. 
Assume that $(K_X+F)^2>0$. 
Then there exists a birational morphism $\upsilon_0:X\to\mathbb{P}^2$ 
such that the linear equivalence class of $F$ is one of the following$:$ 
\begin{align*}
(\mathrm{A})\ \ \ &6\ell-2\sum_{i=1}^8e_i-\sum_{i=9}^{12}e_i	&&\mbox{with }(K_X+F)^2=1,\\
(\mathrm{B}1)\ \ \ &7\ell-3e_1-2\sum_{i=2}^{11}e_i			&&\mbox{with }(K_X+F)^2=2,\\
(\mathrm{B}2)\ \ \ &9\ell-3\sum_{i=1}^8e_i-2e_9-2e_{10}-e_{11}	&&\mbox{with }(K_X+F)^2=2,\\
(\mathrm{C})\ \ \ &13\ell-5e_1-4\sum_{i=2}^{10}e_i			&&\mbox{with }(K_X+F)^2=3.
\end{align*}
Here $\ell$ is the pull-back to $X$ of a line on $\mathbb{P}^2$ 
and $e_i$ is that of a $(-1)$-curve contracted by $\upsilon_0$.

In particular, $f$ has a $(-1)$-section if and only if $\upsilon_0$ gives 
$(\mathrm{A})$ or $(\mathrm{B}2)$. 
Furthermore, $\deg\upsilon_0'(F)\geq\deg\upsilon_0(F)$, 
for any birational morphism $\upsilon_0':X\to\mathbb{P}^2$. 
If the equality sign holds, then the types of 
singularities of $\upsilon_0'(F)$ are the same as $\upsilon_0(F)$'s. 
\end{theorem}
%
%=%=%=%=%=%=%=%=%=%=%=%=%=%=%=%=%=%=%=%=%=%=%=%=%=%=%=%=%=%=%=%=%=%=%=%=%=%=%=%=%=%=%
%

\subsection{Proof of Theorem~\ref{MT}}
Let $X$ be a smooth rational surface and 
$f:X\to\mathbb{P}^1$ a relatively minimal fibration of genus two. 
Assume that $(K_X+F)^2>0$. 
Let $(Y, G)$ be the reduction of $(X,F)$. 
Clearly, $Y$ is not $\mathbb{P}^2$. 
We also have $Y\not=\Sigma_d$ from \cite{Martens96} or $(\ref{eq:Hd})$. 
We firstly determine 
numerical possibilities of \#-minimal models of $(Y, G)$. 
Remark that $(K_Y+G)^2=(K_X+F)^2\leq3$ from Proposition~\ref{Prop:UBAdjDeg}. 
Keep the same notation as in \S\ref{Sec:RA}. 
Then we have the following:

%
%=%=%=%=%=%=%=%=%=%=%=%=%=%=%=%=%=%=%=%=%=%=%=%=%=%=%=%=%=%=%=%=%=%=%=%=%=%=%=%=%=%=%
%
\begin{lemma}\label{NT}
Let $f:X\to\mathbb{P}^1$ be a relatively minimal fibration of genus two 
on a smooth rational surface with $(K_X+F)^2>0$ and 
$(Y,G)$ the reduction of $(X,F)$. 
Then a \#-minimal model $(Y^\#,G^\#)$ of $(Y,G)$ is of general type. 
Furthermore, $(Y^\#,G^\#)$ satisfies one of the following$:$ 
\begin{align}
(a,\breve{b},N,m_1,\ldots, m_7,(K_X+F)^2)			&=(2,0,7,2,\ldots, 2,1),	\label{a2I}\\
(a,\breve{b},N,m_1,\ldots, m_{10},(K_X+F)^2)		&=(2,1,10,2,\ldots, 2,2),	\label{a2II}\\
(a,\breve{b},N,m_1,\ldots, m_7,m_8,m_9,(K_X+F)^2)	&=(4,0,9,3,\ldots, 3,2,2,2),\label{a4}\\
(a,\breve{b},N,m_1,\ldots, m_9,(K_X+F)^2)			&=(6,1,9,4,\ldots, 4,3),	\label{a6}\\
(a,\breve{b},N,m_1,\ldots, m_7,m_8,m_9,(K_X+F)^2)	&=(8,0,9,5,\ldots, 5,4,3,3).\label{a8}
\end{align}
In particular, $0\leq d\leq2$ holds, and the case of $d=1$ induces the other cases 
by performing an elementary transformation. % if necessary. 
\end{lemma}
%
%=%=%=%=%=%=%=%=%=%=%=%=%=%=%=%=%=%=%=%=%=%=%=%=%=%=%=%=%=%=%=%=%=%=%=%=%=%=%=%=%=%=%
%
\begin{proof}
Suppose that $(Y^\#,G^\#)$ is of special type. 
From $(\ref{eq:AdjDegSpEven})$ or $(\ref{eq:AdjDegSpOdd})$, 
we have $(K_Y+G)^2\geq4$ or $3+2(a-3)/(a+1)$ according as $a$ is even or not. 
Furthermore, if $(K_Y+G)^2=3$ holds, then we have $a=3$ and $m_0=\cdots=m_N=2$. 
However, it contradicts $(\ref{Eq:GSqSpS})$ and $(\ref{Eq:RedForm})$. 
Hence, $(Y^\#,G^\#)$ is of general type.

When $a$ is odd, from $(\ref{eq:AdjDegGen})$, we have $(K_Y+G)^2\geq4$, 
which is a contradiction. 
Put $a=2$. 
Then $m_1=\cdots=m_N=2$ and $N=3\breve{b}+7$ from $(\#2)$ and $(\ref{Eq:gGenS})$. 
The assumption that $(K_X+F)^2\geq1$ holds, $(\ref{poG})$ and $(\ref{Eq:RedForm})$ 
imply $N\leq11$. 
Hence we have $(\ref{a2I})$ and $(\ref{a2II})$. 
From now on, we concentrate on the case where $a\geq4$ and $a$ is even. 
By putting $\mathfrak{n}=0$ and $\mathfrak{m}=(a+2)/2$ in Lemma~\ref{frak}, 
we have 
\begin{align}
(K_Y+G)^2\geq 3+2a(\breve{b}-1)/(a+2)+(a-6)/(a+2). \label{bprime1}
\end{align}
Hence we have $a=4$ or $6$ when $\breve{b}\geq1$. 
Since the equality $(K_Y+G)^2=3$ in $(\ref{bprime1})$ with $\breve{b}=1$ holds, 
Lemma~\ref{frak} and $(\ref{Eq:gGenS})$ imply that $(Y^\#, G^\#)$ satisfy $(\ref{a6})$ 
when $a=6$ and $\breve{b}\geq1$. 
If $a=4$ and $\breve{b}\geq1$, then we also have $\breve{b}=1$ and $(K_Y+G)^2=3$, 
though the equality in $(\ref{bprime1})$ does not hold. 
Furthermore, we have $N\leq9$ from $(\ref{poG})$ and $(\ref{Eq:RedForm})$. 
Then $(\#2)$ and $(\ref{Eq:gGenS})$ imply $g\geq30-3N\geq3$, which is absurd.

We assume that $\breve{b}=0$. 
It follows from $(\#2)$ and $(\ref{Eq:gGenS})$ that $m_8\leq a/2$ and $N\geq9$. 
By putting $\mathfrak{n}=7$ and $\mathfrak{m}=(a-2)/2$ in Lemma~\ref{frak}, 
we have $(K_Y+G)^2\geq4$. 
Hence $m_8=a/2$. 
If $m_7\leq a/2$ holds, then we have $(K_Y+G)^2\geq3+(a-4)/a$ 
by putting $\mathfrak{n}=6$ and $\mathfrak{m}=a/2$ in Lemma~\ref{frak}. 
Here $(K_Y+G)^2=3$ implies that $(a,m_6,m_7)=(4,3,2)$ and $N\leq9$. 
In this case, $(\ref{Eq:gGenS})$ implies $g=13-N\geq4$, which is absurd. 
Thus $m_7=(a+2)/2$. 
Then we have $G^2=3-\sum^N_{i=9}m_i$ from $(\ref{Eq:gGenS})$ and $(\ref{Eq:GSqGenS})$. 
%It follows from $(\ref{poG})$ and $(\#2)$ that 
Therefore $N=9$ and $m_9\leq3$ hold. 
From $(\ref{Eq:gGenS})$, we have $(\ref{a4})$ or $(\ref{a8})$ according as $m_9=2$ or $3$.

We show the last statement in Lemma~\ref{NT}. 
For example, we consider $(\ref{a6})$. 
Then $G^\#\sim8\Delta_0+(9+4d)\Gamma$ holds 
and we have $d\leq2$ since $\Delta_0.G^\#=9-4d\geq0$. 
We concentrate on the case of $d=2$. 
Then, $G^\#$ has no singular points on the minimal section $\Delta_0$, 
since the intersection number of the strict transforms of $G^\#$ and $\Delta_0$ 
is also non-negative. 
We consider the fibre $\Gamma_1$ of $Y^\#$ passing through a singular point $p_1$ of $G^\#$. 
Perform the elementary transformation $\tau:(Y^\#, G^\#)\dashrightarrow({Y^\#}', {G^\#}')$ at $p_1$, 
then the composite of $\nu:(Y,G)\to (Y^\#, G^\#)$ and $\tau$ is also a birational morphism 
$\nu':(Y,G)\to ({Y^\#}', {G^\#}')$ contracting the strict transform of $\Gamma_1$ 
instead of the exceptional curve. 
Here ${Y^\#}'=\Sigma_1$ and $({Y^\#}', {G^\#}')$ is a \#-minimal model satisfying $(\ref{a6})$. 
Reversibly, perform the elementary transformation 
$\tau^{-1}:({Y^\#}', {G^\#}')\dashrightarrow(Y^\#, G^\#)$ 
at the point corresponding to the strict transform of $\Gamma_1$, 
then the original one is obtained from the \#-minimal model $({Y^\#}', {G^\#}')$. 
The case of $d=0$ is similar and simpler. 
\end{proof}
%
%=%=%=%=%=%=%=%=%=%=%=%=%=%=%=%=%=%=%=%=%=%=%=%=%=%=%=%=%=%=%=%=%=%=%=%=%=%=%=%=%=%=%
%

Next, we observe genus two fibrations in the view of double coverings according to \cite{Hori77}. 
Remark that the fixed part $Z$ of $|K_X+F|$ is vertical with respect to $f:X\to\mathbb{P}^1$ 
and $H^0(X,K_X+F-Z)\simeq H^0(F,K_F)$ by the restriction map (cf.\ \cite[Lemma~1.1]{KK}). 
Consider the rational map $f\times\Phi_{|K_X+F-Z|}:X\to W^\natural$ of degree two, 
where $W^\natural\simeq\varSigma_0$. 
Stein factorization of the morphism obtained by blowing up at the indeterminacy 
uniquely determines a finite double cover $\pi^\natural:X^\natural\to W^\natural$. 
Then, by \cite[Theorem~1]{Hori77}, 
the branch locus $B^\natural$ has only those singularities of types 
$(0)$, $(\mathrm{I}_k)$, $(\mathrm{II}_k)$, $(\mathrm{III}_k)$, $(\mathrm{IV}_k)$ and $(\mathrm{V})$ 
which are listed in \cite[Lemma~6]{Hori77}. 
Furthermore, 
\begin{align}
%B^\natural\sim6\Delta_0+\left(\sum_k\{n(\mathrm{I}_k)+n(\mathrm{III}_k)\}+n(\mathrm{V})+(K_X+F)^2+2\right)\Gamma
B^\natural\sim6\varDelta_0+(\epsilon+(K_X+F)^2+2)\varGamma,\ \ \ 
\epsilon=\sum_k\{n(\mathrm{I}_k)+n(\mathrm{III}_k)\}+n(\mathrm{V})\label{LECofB}
\end{align}
and 
\begin{align}
(K_X+F)^2=\sum_k\{(2k-1)(n(\mathrm{I}_k)+n(\mathrm{III}_k))+2k(n(\mathrm{II}_k)+n(\mathrm{IV}_k))\}+n(\mathrm{V}),
\label{THindex}
\end{align}
where $n(*)$ denotes the number of singularities of type $(*)$, 
from \cite[Theorem~3]{Hori77}. 
Let $\widetilde{\pi}:\widetilde{X}\to\widetilde{W}$ be the finite double cover obtained from 
the canonical resolution (as in \cite[\S3]{Hori77}) of $\pi^\natural:X^\natural\to W^\natural$. 
Remark that $\sigma:\widetilde{W}\to W^\natural$ factors through the composite 
$\sigma^\natural:W^\flat\to W^\natural$ of the blow-ups at exactly $2(K_X+F)^2$ points 
such that the finite double covering $X^\flat$ of $W^\flat$ branched along $B^\flat$ 
has at most rational double points as its singularities. 
Let $E_i$, $1\leq i\leq2(K_X+F)^2$ be the pull-back to $W^\flat$ 
%$\widetilde{W}$ 
of a $(-1)$-curve contracted by $\sigma^\natural$. 
We denote by $\sigma^\flat$ the composite of blow-ups 
such that $\sigma=\sigma^\natural\circ\sigma^\flat$. 
Then $\Phi_{|\varGamma|}\circ\sigma\circ\widetilde{\pi}:\widetilde{X}\to\mathbb{P}^1$ 
is a genus two fibration which contains exactly $(K_X+F)^2$ disjoint $(-1)$-curves in some fibres. 
In fact, the relatively minimal model is the original fibration $f:X\to\mathbb{P}^1$. 
We denote by $\varphi:\widetilde{X}\to X$ 
the composite of the blow-downs which contract the $(K_X+F)^2$ disjoint $(-1)$-curves. 
These rational maps gives a commutative diagram in Figure~\ref{fig:CD}. 
%%%%%%%%%%%%%%%%%%%%%%%%%%%%%%%%%%%%%%%%%%%%%%%%%%%
%%%%%%%%%%%%%%%%%%%%%%%%%%%%%%%%%%%%%%%%%%%%%%%%%%%
%%%%%%%%%%%%%%%%%%%%%%%%%%%%%%%%%%%%%%%%%%%%%%%%%%%
\begin{figure}[hbtp]
\begin{center}
\begin{picture}(270,170)
\setlength{\unitlength}{1.08358pt}
%\qbezier(0,160)(135,160)(270,160)
%\qbezier(270,160)(270,75)(270,-5)
%\qbezier(0,175)(0,95)(0,-5)
%\qbezier(0,-5)(135,-5)(270,-5)
%
%\qbezier(200,185)(200,170)(200,155)
%\qbezier(200,150)(200,135)(200,120)
%\qbezier(200,115)(200,100)(200,85)
%\qbezier[100](200,85)(199,87.5)(197,88.5)
%\qbezier[100](200,85)(201,87.5)(203,88.5)
%
\put(230,155){\makebox(0,0){$X$}}
\put(235.5,93.5){\makebox(0,0){$f\times\Phi_{|K_X+F-Z|}$}}
\qbezier(227.5,148.253)(218.125,132.015)(208.75,115.777)
\qbezier(205.625,110.365)(196.25,94.127)(186.875,77.889)
\qbezier(183.75,72.476)(174.375,56.238)(165,40)
\qbezier[100](165,40)(165.696,43.206)(164.589,45.289)
\qbezier[100](165,40)(167.429,42.206)(169.786,42.289)
\put(143.5,70){\makebox(0,0){$\pi^\natural$}}
\qbezier(153.5,38.5)(153.5,68.5)(153.5,98.5)
\qbezier[100](153.5,38.5)(154.5,41)(156.5,42)
\qbezier[100](153.5,38.5)(152.5,41)(150.5,42)
\put(123,138){\makebox(0,0){$\varphi$}}
\qbezier(27.5,108.5)(125.128,130.144)(222.76,151.788)
\qbezier[100](222.76,151.788)(220.103,152.223)(218.693,153.959)
\qbezier[100](222.76,151.788)(220.535,150.271)(219.992,148.102)
\put(17.5,106){\makebox(0,0){$\widetilde{X}$}}
\qbezier(25,105)(50,105)(75,105)
\qbezier[100](75,105)(72.5,106)(71.5,108)
\qbezier[100](75,105)(72.5,104)(71.5,102)
\put(85,106.5){\makebox(0,0){$X^\flat$}}
\qbezier(82.5,38.5)(82.5,68.5)(82.5,98.5)
\qbezier[100](82.5,38.5)(83.5,41)(85.5,42)
\qbezier[100](82.5,38.5)(81.5,41)(79.5,42)
\qbezier(95,105)(120,105)(145,105)
\qbezier[100](145,105)(142.5,106)(141.5,108)
\qbezier[100](145,105)(142.5,104)(141.5,102)
\put(155,106.5){\makebox(0,0){$X^\natural$}}
\put(7.5,70){\makebox(0,0){$\widetilde{\pi}$}}
\qbezier(17.5,40)(17.5,69)(17.5,98)
\qbezier[100](17.5,40)(18.5,42.5)(20.5,43.5)
\qbezier[100](17.5,40)(16.5,42.5)(14.5,43.5)
\put(17.5,31.35){\makebox(0,0){$\widetilde{W}$}}
\put(50,43){\makebox(0,0){$\sigma^\flat$}}
\qbezier(25,30)(50,30)(75,30) %50
\qbezier[100](75,30)(72.5,31)(71.5,33)
\qbezier[100](75,30)(72.5,29)(71.5,27)
\put(85,31.5){\makebox(0,0){$W^\flat$}}
\put(120,43){\makebox(0,0){$\sigma^\natural$}}
\qbezier(95,30)(120,30)(145,30) %50
\qbezier[100](145,30)(142.5,31)(141.5,33)
\qbezier[100](145,30)(142.5,29)(141.5,27)
\put(155,31.5){\makebox(0,0){$W^\natural$}}
\qbezier[100](153,23.5)(151.581,21.212)(149.438,20.574)
\qbezier[100](153,23.5)(153.551,20.864)(155.347,19.532)
\qbezier(20,23.5)(20,5.5)(85,5.5)
\qbezier(153,23.5)(153,5.5)(85,5.5)
\put(85,-5){\makebox(0,0){$\sigma$}}
\end{picture}
\caption{}\label{fig:CD}
\end{center}
\end{figure}
%%%%%%%%%%%%%%%%%%%%%%%%%%%%%%%%%%%%%%%%%%%%%%%%%%%
%%%%%%%%%%%%%%%%%%%%%%%%%%%%%%%%%%%%%%%%%%%%%%%%%%%
%%%%%%%%%%%%%%%%%%%%%%%%%%%%%%%%%%%%%%%%%%%%%%%%%%%
The pull-back by $\widetilde{\pi}$ of the strict transform by $\sigma$ of $\epsilon$ fibres 
which form $B^\natural$'s singularities of types $(\mathrm{I}_k)$, $(\mathrm{III}_k)$ and $(\mathrm{V})$ 
are $\epsilon$ double $(-1)$-curves on $\widetilde{X}$. 
Let $\mathcal{E}$ be the sum of the $\epsilon$ $(-1)$-curves. 
In fact, $\varphi(\mathcal{E})$ forms simple base points of $|K_X+F-Z|$. 
From \cite[Lemmas~10 and 13]{Hori77}, we have 
\begin{align*}
(\sigma^\flat\circ\widetilde{\pi})^*(-K_{W^\flat})&\sim\varphi^*(2K_X-({K_X}^2)F),\\
(\sigma\circ\widetilde{\pi})^*\varDelta_0&\sim\varphi^*(K_X+F-Z)-\mathcal{E}
\end{align*}
and $(\sigma\circ\widetilde{\pi})^*\varGamma\sim\varphi^*F$. 
In fact, from Castelnuovo's rationality criterion, the branch divisor satisfies 
\begin{align}
\label{eq:Rationality}
h^0(W^\flat, -K_{W^\flat}+{K_X}^2( \sigma^\natural )^*\varGamma )=0.
\end{align}
Conversely, for a desired value of $(K_X+F)^2$ or $K_X^2$, we obtain a relatively minimal fibration 
$f:X\to\mathbb{P}^1$ of genus two curves from a reducible divisor $B^\natural$ on 
$W^\natural:=\mathbb{P}^1\times\mathbb{P}^1$ with $(\ref{LECofB})$ and $(\ref{THindex})$ 
%$W^\natural:=\varSigma_0$ with $(\ref{LECofB})$ and $(\ref{THindex})$ 
by taking the finite double covering $\pi^\natural:X^\natural\to W^\natural$ 
whose branch divisor is $B^\natural$. 
Then $(\ref{eq:Rationality})$ implies that $X$ is rational from Castelnuovo's rationality criterion.

%
%=%=%=%=%=%=%=%=%=%=%=%=%=%=%=%=%=%=%=%=%=%=%=%=%=%=%=%=%=%=%=%=%=%=%=%=%=%=%=%=%=%=%
%
\begin{lemma}\label{Lem:a8}
A \#-minimal model satisfying $(\ref{a8})$ does not occur. 
\end{lemma}
%
%=%=%=%=%=%=%=%=%=%=%=%=%=%=%=%=%=%=%=%=%=%=%=%=%=%=%=%=%=%=%=%=%=%=%=%=%=%=%=%=%=%=%
%
\begin{proof}
Suppose that there exists a relatively minimal fibration $f:X\to\mathbb{P}^1$ of genus two 
on a smooth rational surface such that a \#-minimal model $(Y^\#, G^\#)$ 
of $(X,F)$ satisfies $(\ref{a8})$. 
Remark that singular points $p_1,\ldots,p_8$ of $G^\#$ 
are not any infinitely near point of the triple point $p_9$. 
Let $e_{10}$ be the $(-1)$-curve corresponding by $\nu$ to $p_9$. 
In particular, the strict transform $\widehat{e_{10}}$ to $\widetilde{X}$ of $e_{10}$ is 
also smooth and irreducible. 
By applying the projection formula, we have 
\begin{align}
(\sigma^\flat\circ\widetilde{\pi})_*\widehat{e_{10}}.(-K_{W^\flat})=
\widehat{e_{10}}.\varphi^* (2K_X+F)=e_{10}.(2K_X+F)=1.\label{e10-K}
\end{align}
We remark that $e_{10}$ is not a component of any fibre of $f$. 
Therefore, we have 
\begin{align*}
(\sigma\circ\widetilde{\pi})_*\widehat{e_{10}}.\varDelta_0=e_{10}.(K_X+F)-e_{10}.Z-\widehat{e_{10}}.\mathcal{E}\leq e_{10}.(K_X+F)=2
\end{align*}
by applying the projection formula. 
Furthermore, $(\sigma\circ\widetilde{\pi})_*\widehat{e_{10}}.\varGamma=3$ holds. 
This implies that $\sigma\circ\widetilde{\pi}$ is birational on $\widehat{e_{10}}$. 
Therefore, $(\sigma\circ\widetilde{\pi})_*\widehat{e_{10}}$ is also irreducible and reduced. 
We also have $(\sigma\circ\widetilde{\pi})_*\widehat{e_{10}}\sim3\varDelta_0+\varGamma$ or 
$3\varDelta_0+2\varGamma$. 
In the former case, 
%$(\sigma^\flat\circ\widetilde{\pi})_*\widehat{e_{10}}\sim3(\sigma^\natural)^*\Delta_0+(\sigma^\natural)^*\Gamma-\sum_{i=1}^6m_iE_i$ 
\begin{align*}
(\sigma^\flat\circ\widetilde{\pi})_*\widehat{e_{10}}\sim3(\sigma^\natural)^*\varDelta_0+(\sigma^\natural)^*\varGamma-\sum_{i=1}^6m_iE_i
\end{align*}
with $m_i=0$ or $1$. 
Then $(\sigma^\flat\circ\widetilde{\pi})_*\widehat{e_{10}}.(-K_{W^\flat})=8-\sum_{i=1}^6m_i\geq2$, 
which contradicts $(\ref{e10-K})$. 
Next, consider the latter case, and put 
%$(\sigma^\flat\circ\widetilde{\pi})_*\widehat{e_{10}}\sim3(\sigma^\natural)^*\Delta_0+2(\sigma^\natural)^*\Gamma-\sum_{i=1}^6m_iE_i$ 
\begin{align*}
(\sigma^\flat\circ\widetilde{\pi})_*\widehat{e_{10}}\sim3(\sigma^\natural)^*\varDelta_0+2(\sigma^\natural)^*\varGamma-\sum_{i=1}^6m_iE_i
\end{align*}
with $0\leq m_i\leq2$. 
Since the arithmetic genus of $(\sigma\circ\widetilde{\pi})_*\widehat{e_{10}}$ is two, 
$\#\{m_i|m_i=2\}=2$ holds. 
This implies $(\sigma^\flat\circ\widetilde{\pi})_*\widehat{e_{10}}.(-K_{W^\flat})=10-\sum_{i=1}^6m_i\geq2$, 
which is also a contradiction to $(\ref{e10-K})$. 
\end{proof}

For a relatively minimal fibration $f:X\to\mathbb{P}^1$ of curves of genus two 
on a smooth rational surface with $(K_X+F)^2>0$, 
we can take a \#-minimal model $(Y^\#, G^\#)$ of the reduction $(Y, G)$ 
in order that $Y^\#=\Sigma_1$ from Lemma~\ref{NT}. 
Let $\nu_0:Y\to\mathbb{P}^2$ be the composite of $\nu:Y \to Y^\#$ and 
the blow-down $Y^\#\to\mathbb{P}^2$ contracting the minimal section $\Delta_0$. 
We denote by $\upsilon_0:X\to\mathbb{P}^2$ 
the composite of the natural morphism $\mu:X \to Y$ and $\nu_0:Y\to\mathbb{P}^2$. 
When $(Y^\#, G^\#)$ satisfies $(\ref{a2I})$, $(\ref{a2II})$, $(\ref{a4})$ and $(\ref{a6})$ respectively, 
$\upsilon_0$ is a birational morphism giving $(\mathrm{A})$, $(\mathrm{B}1)$, $(\mathrm{B}2)$ 
and $(\mathrm{C})$ in {\upshape Theorem~\ref{MT}}. 
Conversely, a birational morphism $\upsilon_0:X\to\mathbb{P}^2$ as in {\upshape Theorem~\ref{MT}} 
gives a \#-minimal model of the reduction of $(X, F)$ by blowing $(\upsilon_0(X), \upsilon_0(F))$ 
up at a singular point of $\upsilon_0(F)$ with a maximal multiplicity.

\begin{lemma}\label{Barabara}
When $\upsilon_0:X\to\mathbb{P}^2$ gives $(\mathrm{B}1)$ and $(\mathrm{C})$ 
in {\upshape Theorem~\ref{MT}} respectively, $F.C\geq2$ and $4$ for all $(-1)$-curves $C$ on $X$. 
In particular, $f:X\to\mathbb{P}^1$ has a $(-1)$-section if and only if 
$\upsilon_0$ gives $(\mathrm{A})$ or $(\mathrm{B}2)$ in {\upshape Theorem~\ref{MT}}. 
Furthermore, for a given $f$, all of \#-minimal models of the reduction 
satisfy one of the four $(\ref{a2I})$, $(\ref{a2II})$, $(\ref{a4})$ and $(\ref{a6})$. 
\end{lemma}
\begin{proof}
Assume that $\upsilon_0:X\to\mathbb{P}^2$ gives $(\mathrm{C})$ in Theorem~\ref{MT}. 
Since $\upsilon_0(e_1)$ is not any infinitely near point of 
another singular point of $\upsilon_0(F)$, we have a base-point-free pencil $|\ell-e_1|$, 
where $\ell$ and $e_1$ denote the same in Theorem~\ref{MT}. 
It implies that $(\ell-e_1).C\geq0$ for all $(-1)$-curves $C$ on $X$. 
Hence, we have $F.C=((\ell-e_1)-4K_X).C\geq-4K_X.C=4$. 
Therefore, $f$ indeed have no $(-1)$-section. 
The case of $(\mathrm{B}1)$ in Theorem~\ref{MT} is the quite same argument. 
When $\upsilon_0$ gives $(\mathrm{B}2)$ in Theorem~\ref{MT}, 
$e_{11}$ must be irreducible, since $F.(e_{11}-e_i)<0$ for all $i<11$. 
%since $F.e_i>F.e_{11}$ for all $i<11$. 
Thus $e_{11}$ is a $(-1)$-section of $f$. 
In the same way, $f$ has at least one $(-1)$-section 
when $\upsilon_0$ gives $(\mathrm{A})$ in Theorem~\ref{MT}.

Consider $\rho(X)$. 
From Lemmas~\ref{NT} and \ref{Lem:a8}, 
all of \#-minimal models of the reduction of $(X,F)$ with $(K_X+F)^2=1$ and $3$ 
satisfy $(\ref{a2I})$ and $(\ref{a6})$, respectively. 
Furthermore, in the case of $(K_X+F)^2=2$, a \#-minimal model satisfies 
$(\ref{a4})$ or $(\ref{a2II})$ according as $f$ has a $(-1)$-section or not. 
\end{proof}

Remark that a \#-minimal model of the reduction of $(X, F)$ is not unique in general 
for a relatively minimal fibration $f:X\to\mathbb{P}^1$ of genus two 
on a rational surface with $(K_X+F)^2>0$. 
Hence a birational morphism $\upsilon_0:X\to\mathbb{P}^2$ as before Lemma~\ref{Barabara} 
is also not unique. 
However, from Lemma~\ref{Barabara}, the linear equivalence class of $F$ is one of the four 
in Theorem~\ref{MT} by %taking 
any birational morphism $\upsilon_0:X\to\mathbb{P}^2$ 
which gives a \#-minimal model of the reduction of $(X, F)$.

%$\Gamma_Y$ でなく $\ell-e_1$ を使う。
% $(\mathrm{A})$ や $(\mathrm{B}2)$ を考えるとこれでよい。
Let $\Gamma_X$ be the pull-back of $\Gamma_Y$ to $X$. 
In fact, $|\Gamma_X|$ is that of $|\Gamma|$ on $Y^\#$. 
Indeed, the base-point-free pencil $|\Gamma_X|$ of rational curves 
has a minimality for $F$ as follows:

\begin{lemma}\label{MinOfAdjDeg}
Keep the notation as above. 
Let $R$ be a general member of a base-point-free pencil of rational curves on $X$. 
Then $F.R \geq F.\Gamma_X$. 
Furthermore, $F.R=F.\Gamma_X$ if and only if $R \sim \Gamma_X$ 
when $\upsilon_0$ gives $(\mathrm{B}1)$ or $(\mathrm{C})$ in {\upshape Theorem~\ref{MT}}. 
\end{lemma}
\begin{proof}
Let $\upsilon_0:X\to\mathbb{P}^2$ be a birational morphism giving $(\mathrm{C})$ in Theorem~\ref{MT}. 
From the assumption $R^2=0$, 
we have $F.R=(-4K_X+\Gamma_X).R=8+\Gamma_X.R\geq8=F.\Gamma_X$. 
Here $F.R=F.\Gamma_X$ if and only if $\Gamma_X.R=0$, which implies that 
a fibre of the ruling defined by $|\Gamma_X|$ contains $R$. 
Hence we have $R \sim \Gamma_X$. 
We see the case of $(\mathrm{B}1)$ in Theorem~\ref{MT} by the quite same argument. 
The other cases are also similar. 
\end{proof}

We finally show the last two statements in Theorem~\ref{MT} as follows:  
Let $\upsilon_0:X\to\mathbb{P}^2$ be a birational morphism as in {\upshape Theorem~\ref{MT}}. 
We restrict ourselves to the case of $(\mathrm{C})$, 
since the other cases are similar and simpler. 
At first, we suppose that there exists a birational morphism $\upsilon_0':X\to\mathbb{P}^2$ 
%At first, we suppose that there exists a birational morphism $\upsilon_0':X\to(\mathbb{P}^2)'$ 
causing $\deg\upsilon_0'(F)\leq12$. 
Remark that a linear projection whose centre is a singular point of $\upsilon_0'(F)$ 
induces a ruling on $X$. 
Therefore, multiplicities of singular points of $\upsilon_0'(F)$ are at most 
$(\deg\upsilon_0'(F)-8)$ from Lemma~\ref{MinOfAdjDeg}. 
We have a \#-minimal model which is different from $(\ref{a6})$ 
by blowing $(\upsilon_0'(X), \upsilon_0'(F))$ up at a singular point of $\upsilon_0'(F)$. 
It contradicts Lemma~\ref{Barabara}.

Next, we assume that there exists another birational morphism 
$\upsilon_0':X\to\mathbb{P}^2$ making $\deg\upsilon_0'(F)=13$. 
%$\upsilon_0':X\to(\mathbb{P}^2)'$ making $\deg\upsilon_0'(F)=13$. 
In the same way as above, multiplicities of singular points of $\upsilon_0'(F)$ are at most five. 
On the other hand, from Lemma~\ref{Barabara}, they are at least four. 
Furthermore, $\rho(X)=11$ implies 
that the number of singular points of $\upsilon_0'(F)$ is also exactly ten. 
Since the genus of $F$ is two, 
singularities of $\upsilon_0'(F)$ must be the same as $\upsilon_0(F)$'s by considering $(\ref{Eq:gGenS})$.

%
%=%=%=%=%=%=%=%=%=%=%=%=%=%=%=%=%=%=%=%=%=%=%=%=%=%=%=%=%=%=%=%=%=%=%=%=%=%=%=%=%=%=%
%

\smallskip

This completes the proof of Theorem~\ref{MT}. 
\hspace*{\fill}$\square$

\begin{remark}
Let $f:X\to\mathbb{P}^1$ be a fibred rational surface of genus two and 
$\widetilde{X}$ as before {\upshape Lemma~\ref{Lem:a8}}, 
which is uniquely determined by the relative canonical map. 
Assume that $\rho(X)\not=14$ or Mordell-Weil group of $f$ is non-trivial. 
Then, from {\upshape Theorem~\ref{MT}} and \cite{Kitagawa5}, 
$\widetilde{X}$ can be obtained by blowing $\mathbb{P}^2$ up at thirteen points. 
\end{remark}
%
%=%=%=%=%=%=%=%=%=%=%=%=%=%=%=%=%=%=%=%=%=%=%=%=%=%=%=%=%=%=%=%=%=%=%=%=%=%=%=%=%=%=%
%

%\subsection{Birational morphisms with a minimality for a general fibre}\label{Sec:OM}
%\subsection{Minimal models for pencils}\label{Sec:OM}
%\subsection{Better \#-minimal models}\label{Sec:OM}
\subsection{\#-minimal models}\label{Sec:OM}
%\subsection{Minimal models of pencils}\label{Sec:OM}
%\subsection{Models with certain minimalities}\label{Sec:OM}
%\subsection{Other models}\label{Sec:OM}
Let $f:X\to\mathbb{P}^1$ be as in Theorem~\ref{MT} and $(Y, G)$ the reduction of $(X, F)$. 
By contracting a $(-1)$-curve whose intersection number with $G$ is the smallest among 
all of $(-1)$-curves on $Y$, we have a pair $(Y_1, G_1)$ as the image of $(Y, G)$. 
By performing the same procedure for $(Y_1, G_1)$, we get a pair $(Y_2, G_2)$. 
In this way, $Y_{(\rho(Y)-2)}$ becomes a relatively minimal model $\Sigma_d$ of a ruling on $X$ 
such that the image of $F$ is simple in the view of its singular points, though there exist 
many rulings on $X$ and relatively minimal models in general.

In fact, such a model is essentially a \#-minimal model of $(Y, G)$ as follows: 
Let $\nu:Y\to\Sigma_d$ be the birational morphism giving the above model. 
When $d=0$ and $\Gamma.\nu_*G>\Delta_0.\nu_*G$, for simplicity, 
change the considered ruling for the other one. 
This is non-essential since contracted $(-1)$-curves are kept. 
First of all, we remark that $\nu_*G$ has singularities. 
Let $p_i$ be the singular points of $\nu_*G$ including infinitely near ones, 
and let $m_i$ be the multiplicity of $\nu_*G$ at $p_i$. 
We may assume that $m_i \geq m_{i+1}$ for simplicity. 
Put $a=\nu^*\Gamma.(K_Y+G)=\Gamma.(K_{\Sigma_d}+\nu_*G)$ and $b=(a+2)d+\Delta_0.\nu_*G$. 
Then $(\#1)$ is satisfied. 
We suppose that $(\Gamma.\nu_*G)/2<m_1$, which induces a contradiction as follows: 
We consider the fibre $\Gamma_1$ of $\Sigma_d$ 
passing through a singular point $p_1$ of $\nu_*G$. 
Furthermore, we perform the elementary transformation 
$\tau:\Sigma_d\dashrightarrow \Sigma_{d+1}$ or $\Sigma_{d-1}$ at $p_1$ 
according as $p_1$ lies on $\Delta_0$ or not. 
Then the composite $\tau\circ\nu:Y\to\Sigma_{d+1}$ or $\Sigma_{d-1}$ is also 
a birational morphism contracting the strict transform $\widehat{\Gamma_1}$ of $\Gamma_1$ 
instead of the exceptional curve. 
The intersection number of $\widehat{\Gamma_1}$ and the image of $G$ is 
$(\nu_*G.\Gamma-m_1)$, which contradicts the procedure of $\nu$. 
Similarly, we have $m_1\leq b-(a+2)$ when $d=1$. 
Hence, $(\Sigma_d, \nu_*G)$ must be a \#-minimal model of $(Y,G)$. 
From Lemmas~\ref{NT}, \ref{Lem:a8} and \ref{Barabara}, it satisfies 
one of the four $(\ref{a2I})$, $(\ref{a2II})$, $(\ref{a4})$ and $(\ref{a6})$. 
In fact, the converse holds as follows:

%	特に、線織構造に対して水平な $(-1)$-curve が問題、 
%	$F$ との交わりが小さいものが現れないとは言い切れない。
%	ひょっとすると、＃極小モデルの線織構造の極小性の話に行き着くか？　
%	例えば、四重点の一位の無限近傍点として二重点が二つのったり、
%	特異点の配置も複雑で、特異点の個数も増えてくると、ますます混み入る。
%	今回の証明を思い出すと、 $(-1)$-curve が、
%	一般ファイバーの像から反多重標準因子を、
%	有効性を保ちながら引けるだけ引いていった残りの因子の固定成分に含まれることが、
%	双有理射の極小性が崩れる必要十分条件となってはいないだろうか。
%	もしそうだとして、それが存在したら、$|F|$ の固定成分となりはしないか。
%	こうして背理法から、晴れて肯定できまいか。
%	10/6 今野先生からも、逆は非自明であるというお墨付きを頂く。

%By putting the minimality of birational morphisms into the \#-minimality of the images, 
%we show the first and the third statements in Theorem~\ref{PlanePlan} as follows: 

\begin{proposition}\label{Minimini}
Let $f:X\to\mathbb{P}^1$ denote a fibred rational surface of genus two. 
Assume that $(K_X+F)^2>0$. 
Let $(Y, G)$ be the reduction of $(X, F)$ and $\nu:Y\to\Sigma_d$ a birational morphism. 
When $d=0$ and $\Gamma.\nu_*G>\Delta_0.\nu_*G$, 
change the considered ruling for the other one. 
Then, $\nu$ is associated with the procedure as at the start of {\upshape \S\ref{Sec:OM}} 
if and only if 
$(\Sigma_d, \nu_*G)$ is a \#-minimal model of $(Y, G)$, which satisfies 
one of the four $(\ref{a2I})$, $(\ref{a2II})$, $(\ref{a4})$ and $(\ref{a6})$. 
In particular, $d\leq2$ and 
any birational morphism $\upsilon_0:X\to\mathbb{P}^2$ in {\upshape Theorem~\ref{IntMT}} 
or $\ref{MT}$ contracts step by step a $(-1)$-curve whose intersection number with $F$ 
is the smallest among all $(-1)$-curves. 
\end{proposition}
\begin{proof}
We recall the proof of the last statement in Lemma~\ref{NT}. 
Then we only have to show the case of $d=1$ in Proposition~\ref{Minimini}, 
that is, the last statement.

When $\upsilon_0$ gives $(\mathrm{B}1)$ or $(\mathrm{C})$, 
$\nu$ is associated with the procedure as at the start of {\upshape \S\ref{Sec:OM}} 
from Lemma~\ref{Barabara}. 
When $\upsilon_0$ gives $(\mathrm{A})$ or $(\mathrm{B}2)$, 
the definition of the reduction implies that $C.G\geq2$ for any $(-1)$-curve $C$ on $Y$. 
Hence, $\upsilon_0$ giving $(\mathrm{A})$ is associated with the procedure. 
When $\upsilon_0$ gives $(\mathrm{B}2)$, two $(-1)$-curves $e_{10}$ and $e_9$ 
as in Theorem~\ref{MT} corresponds to double points of $\upsilon_0(F)$. 
Let $(Z, H)$ denote the image of $(Y, G)$ by contracting only the two. 
From the genus formula, $D.H\geq3$ for any $(-1)$-curve $D$ on $Z$ since $H\sim-3K_Z$. 
Thus it also does so. 
\end{proof}

Furthermore, by comparing all of \#-minimal models of the reductions, 
we develop Lemma~\ref{NT} as follows:

\begin{theorem}\label{PlanePlan}
Let $f:X\to\mathbb{P}^1$ denote a fibred rational surface of genus two. 
Assume that $\rho(X)\leq13$. 
Let $(Y^\#, G^\#)$ be a \#-minimal model of the reduction $(Y, G)$ of $(X, F)$ and 
$d$ the degree of $Y^\#$ as a relatively minimal model $\Sigma_d$. 
Then $d\leq2$, and the following holds. 
\begin{itemize}
\item[$(0)$]
In the cases $(\mathrm{B}1)$ and $(\mathrm{C})$ in {\upshape Theorem~\ref{IntMT}}, 
any $f:X\to\mathbb{P}^1$ admits $d=0$. 
Otherwise, 
there exists $f:X\to\mathbb{P}^1$ not admitting $d=0$. 
\item[$(1)$]
All of $f:X\to\mathbb{P}^1$'s admit $d=1$. 
\item[$(2)$]
For each case in {\upshape Theorem~\ref{IntMT}}, 
there exists $f:X\to\mathbb{P}^1$ not admitting $d=2$. 
\end{itemize}
\end{theorem}

In order to show the cases $(\mathrm{B}1)$ and $(\mathrm{C})$ for $d=2$, we consider 
a $(-2)$-{\it section} of $f$, i.e., a $(-2)$-curve %on $X$ 
whose intersection number with $F$ is equal to one. 
As a sufficient condition of the existence of it, we have the following:

\begin{lemma}\label{Lem:suff}
Let $f:X\to\mathbb{P}^1$ be a fibred rational surface which has a birational morphism 
$\upsilon_0:X\to\mathbb{P}^2$ giving 
$(\mathrm{B}1)$ or $(\mathrm{C})$ in {\upshape Theorem~\ref{MT}}. 
If $(X, F)$ admits a \#-minimal model $(\Sigma_2, \nu_*F)$, 
then the pull-back to $X$ of the minimal section $\Delta_0$ of $\Sigma_2$ 
is a $(-2)$-section of $f$. 
\end{lemma}
\begin{proof}
When $\upsilon_0$ gives $(\mathrm{B}1)$ and $(\mathrm{C})$ respectively, 
$(\Sigma_2, \nu_*F)$ satisfies $(\ref{a2II})$ and $(\ref{a6})$ from 
Lemmas~\ref{NT}, \ref{Lem:a8} and \ref{Barabara}. 
Since $\Delta_0.\nu_*F=1$, any singular point of $\nu_*F$ is not on $\Delta_0$. 
Therefore, the pull-back to $X$ of $\Delta_0$ is a $(-2)$-section of $f$. 
\end{proof}

A necessary condition of the existence of a $(-2)$-section of $f$ is as follows:

\begin{lemma}\label{Lem:mts}
Let $f:X\to\mathbb{P}^1$ be a fibred rational surface of genus two 
and $D$ a $(-2)$-section of $f$. 
Then the image of $D$ by $f\times\Phi_{|K_X+F-Z|}:X\dashrightarrow W^\natural$ as before 
{\upshape Lemma~\ref{Lem:a8}} is a section of $W^\natural$ which is linearly equivalent to 
$\varDelta_0+\varGamma$ or $\varDelta_0$ according as $D$ meets the fixed locus of $|K_X+F|$ or not. 
Furthermore, for each case, among the $2(K_X+F)^2$ points which correspond to $(-1)$-curves 
contracted by $\sigma^\natural:W^\flat\to W^\natural$, 
the number of points lying on the image of $D$ is the following$:$

When $(K_X+F)^2=1$, it passes through one of the two points in the case where 
it is linearly equivalent to $\varDelta_0+\varGamma$, and the latter case does not occur. 

When $(K_X+F)^2=2$, it passes through two of the four points in the former case, 
and it does not pass through any point of them in the latter case. 

When $(K_X+F)^2=3$, it is three and one respectively, in the former and latter case. 
\end{lemma}
\begin{proof}
Let $D$ be a $(-2)$-section of $f$. 
The genus formula implies $K_X.D=0$. 
We use the notation as before Lemma~\ref{Lem:a8}. 
Remark that the strict transform $\widehat{D}$ to $\widetilde{X}$ of $D$ is 
also smooth and irreducible. 
In the way similar to the proof of Lemma~\ref{Lem:a8}, we have that 
\begin{align*}
(\sigma\circ\widetilde{\pi})_*\widehat{D}.\varDelta_0 
=D.(K_X+F)-D.Z-\widehat{D}.\mathcal{E}\leq D.(K_X+F)=1
\end{align*}
and $(\sigma\circ\widetilde{\pi})_*\widehat{D}.\varGamma =D.F=1$. 
Thus $(\sigma\circ\widetilde{\pi})_*\widehat{D}\sim\varDelta_0$ or $\varDelta_0+\varGamma$. 
Furthermore, 
\begin{align*}
(\sigma^\flat\circ\widetilde{\pi})_*\widehat{D}.(-K_{W^\flat})=D.(2K_X-({K_X}^2)F)=-{K_X}^2
\end{align*}
implies the statements for the $2(K_X+F)^2$ points 
which correspond to the contracted $(-1)$-curves. 
\end{proof}

Similarly, we have the following:

\begin{lemma}\label{BoneBtwo}
Let $f:X\to\mathbb{P}^1$ be a fibred rational surface of genus two. 
Assume that $(K_X+F)^2=2$. 
Let $f\times\Phi_{|K_X+F-Z|}:X\dashrightarrow W^\natural$ denote the double cover 
branched along $B^\natural$ as before {\upshape Lemma~\ref{Lem:a8}}. 
Then, there exists a birational morphism $\upsilon_0:X\to\mathbb{P}^2$ such that 
$\upsilon_0(F)$ is as in $(\mathrm{B}2)$ in {\upshape Theorem~\ref{IntMT}} if and only if 
$B^\natural$ include exactly one minimal section of $W^\natural$, 
which induces the unique $(-1)$-section of $f$. 
\end{lemma}
\begin{proof}
Let $f:X\to\mathbb{P}^1$ be a fibred rational surface of genus two. 
Assume that $(K_X+F)^2=2$. 
Then the branch divisor $B^\natural$ includes at most one minimal section of $W^\natural$ 
from $(\ref{eq:Rationality})$.

Firstly, we assume that there exists $\upsilon_0:X\to\mathbb{P}^2$ giving $(\mathrm{B}2)$. 
Remark that the strict transform $\widehat{e_{11}}$ to $\widetilde{X}$ of 
the unique $(-1)$-section $e_{11}$ of $f$ is smooth and irreducible. 
By the quite same argument in Lemma~\ref{Lem:mts}, 
$(\sigma\circ\widetilde{\pi})_*\widehat{e_{11}}$ is a minimal section $\varDelta_\infty$ of $W^\natural$, 
and $\varDelta_\infty$ passes through two of the four points which correspond to $(-1)$-curves 
contracted by $\sigma^\natural$. 
Consider singularities of $B^\natural$ on $\varDelta_\infty$. 
Then $B^\natural$ must include $\varDelta_\infty$ from $(\ref{LECofB})$ and $(\ref{THindex})$. 
%こんな風にはしょらざるを得ない

Next, we return the first situation and assume that $B^\natural$ includes a minimal section 
$\varDelta_\infty$ of $W^\natural$. 
We only have to show that the strict transform of $\varDelta_\infty$ by $\sigma\circ\widetilde{\pi}$ 
is the pull-back to $\widetilde{X}$ of the double $(-1)$-section of $f$. 
Consider singularities of $B^\natural$ on $\varDelta_\infty$. 
Then, from $(\ref{LECofB})$ and $(\ref{THindex})$, any singularities of types $(\mathrm{III}_k)$, 
$(\mathrm{IV}_k)$ and $(\mathrm{V})$ does not occur, and singularities of $B^\natural$ except 
of type $(0)$ are two of types $(\mathrm{I}_1)$ or one of type $(\mathrm{II}_1)$. 
Furthermore, $\varDelta_\infty$ and other components of $B^\natural$ except fibres of $W^\natural$ 
form two simple triple points or one two-fold triple point of $B^\natural$ except the fibres 
according to above types of the singularities. 
In particular, the strict transform of $\varDelta_\infty$ by $\sigma^\natural:W^\flat\to W^\natural$ 
is an isolated component whose self-intersection number is $(-2)$ of $B^\flat$. 
Thus a standard calculation for a double cover leads the desired statement. 
\end{proof}

A genus two fibration on a rational surface is obtained from 
a finite double cover of $W^\natural=\mathbb{P}^1\times\mathbb{P}^1$. 
In particular, a branch divisor $B^\natural$ on $W^\natural$ with $(\ref{LECofB})$, 
$(\ref{THindex})$ and $(\ref{eq:Rationality})$ uniquely determines 
a fibred rational surface $f:X\to\mathbb{P}^1$ of genus two 
in the way as before Lemma~\ref{Lem:a8}. 
We also use the notation as before Lemma~\ref{Lem:a8}. 
Let $\varGamma_p$ be the fibre of $p\in\mathbb{P}^1$ 
by the first projection $pr_1:W^\natural\to\mathbb{P}^1$ and 
$\varDelta_q$ the fibre of $q\in\mathbb{P}^1$ by the second projection $pr_2$. 
We take $(t,x)$ as an affine coordinate on $W^\natural\setminus(\varDelta_\infty\cup\varGamma_\infty)$.

\begin{example}\label{Ex:Bonemts}
Let $A$ be Zariski closure on $W^\natural$ of the divisor defined by 
\begin{align*}
x^4-2x^2t-2x^3t+x^4t+t^2+2xt^2-x^2t^2+t^3=0, 
\end{align*}
which is irreducible. 
The singularities of $A$ are on $(0,0)$ and on $(\infty,\infty)$. 
Take the section $C$ of $pr_2$ defined by $x^2-t-xt=0$. 
The four-fold double point $(0,0)$ of $A$ has a contact of order eight with $C$. 
%The two-fold double point $(\infty,\infty)$ of $A$ 
%haa a contact of order two with $C$. 
Put $B^\natural=A+C$, which is linearly equivalent to $6\varDelta_0+4\varGamma$. 
In fact, the singularities of $B^\natural$ are 
of type $(\mathrm{IV}_1)$ on $\varGamma_0$ and of type $(0)$ on $\varGamma_\infty$. 
%In particular, $B^\natural$ is reduced. 
We take the finite double cover of $W^\natural$ branched along $B^\natural$ 
and the canonical resolution. 
Then we also see $(\ref{eq:Rationality})$ 
since $B^\natural$ does not include any minimal section of $pr_1$. 
Furthermore, from Lemma~\ref{BoneBtwo}, we have a fibred rational surface $f:X\to\mathbb{P}^1$ 
which have a birational morphism $\upsilon_0:X\to\mathbb{P}^2$ giving $(\mathrm{B}1)$.

Any section of $pr_1$ which is linearly equivalent to $\varDelta_0+\varGamma$ 
passes through at most one of the four points 
which correspond to the $(-1)$-curves contracted by $\sigma^\natural:W^\flat \to W^\natural$. 
Since minimal sections of $pr_1$ not passing through any of the four points 
must meet $C$ at a smooth point of $B^\natural$, 
the strict transforms to $\widetilde{X}$ are irreducible. 
In particular, any of them is not a section of $f$. 
Since $\varDelta_\infty$ meets $A$ at the smooth point $(-1, \infty)$ of $B^\natural$, 
the strict transform to $\widetilde{X}$ is also irreducible. 
Therefore, $f$ has no $(-2)$-section from Lemma~\ref{Lem:mts}. 
\end{example}

\begin{example}\label{Ex:Cmts}
Let $A$ be Zariski closure on $W^\natural$ of the divisor defined by 
\begin{align*}
x^4+2x^2t+4x^3t-4x^4t+t^2+4xt^2-4x^2t^2-4t^3=0, 
\end{align*}
which is irreducible. 
The singularities of $A$ are on $(0,0)$ and on $(\infty,\infty)$. 
$A$ is tangent to $\varGamma_1$ at $(1, 1)$, and passes through $(1/4, 0)$ and $(1/4, \infty)$. 
The bisection $C$ of $pr_1$ or of $pr_2$ defined by 
$x^2+t+2xt-4t^2=0$ has a double point at $(\infty,\infty)$. 
Furthermore, $C$ passes through $(1/4, 0)$ and $(1, 1)$. 
In addition, the four-fold double point $(0,0)$ of $A$ has 
a contact of order eight with $C$. 
Put $B^\natural=A+C+\varGamma_\infty$, 
which is linearly equivalent to $6\varDelta_0+6\varGamma$. 
In fact, the singularities of $B^\natural$ are of type $(\mathrm{IV_1})$ on $\varGamma_0$, 
of type $(\mathrm{V})$  on $\varGamma_\infty$, 
of type $(0)$ on $\varGamma_{1/4}$ and of type $(0)$ on $\varGamma_1$. 
%In particular, $B^\natural$ is reduced. 
If a section of $pr_2$ which is linearly equivalent to $2\varDelta_0+\varGamma$ 
passes through five points $(0,0)$, $(\infty,\infty)$ and the three 
infinitely near points of $(0,0)$ corresponding to double points of $A$, 
then it is unique and is defined by $x^2+2xt+t=0$. 
The section of $pr_2$ meets $\varGamma_\infty$ transversally. 
Therefore, $(\ref{eq:Rationality})$ also holds when 
we take the finite double cover of $W^\natural$ branched along $B^\natural$ 
and the canonical resolution. 
Hence, we obtain a fibred rational surface $f:X\to\mathbb{P}^1$ of genus two with $(K_X+F)^2=3$.

Any section of $pr_1$ which is linearly equivalent to $\varDelta_0+\varGamma$ 
passes through at most two of the six points 
which correspond to the $(-1)$-curves contracted by $\sigma^\natural:W^\flat \to W^\natural$. 
$\varDelta_0$ meets $B^\natural$ at a smooth point $(\infty, 0)$ of $B^\natural$ transversally. 
$\varDelta_\infty$ also meets $B^\natural$ at a smooth point $(1/4, \infty)$ transversally. 
Therefore, the strict transforms to $\widetilde{X}$ are irreducible. 
In particular, they meet $\varphi^*F$ at two points. 
Thus, $f$ has no $(-2)$-section from Lemma~\ref{Lem:mts}. 
\end{example}

From Lemma~\ref{Lem:suff}, Examples~\ref{Ex:Bonemts} and \ref{Ex:Cmts}, 
we have the following:

\begin{lemma}\label{Lem:BoneCStwo}
In each of the cases $(\mathrm{B}1)$ and $(\mathrm{C})$ in {\upshape Theorem~\ref{IntMT}}, 
there exists a fibred rational surface $f:X\to\mathbb{P}^1$ not admitting $\Sigma_2$ 
as the surface of a \#-minimal model of $(X, F)$. 
\end{lemma}

%浪川・上野の分類を使うと特異点 $\upsilon_0(e_i)$ の $\mathbb{P}^2$ における配置に関して, 
%Lemma~\ref{Lem:BoneCSzero} より強い次の結果を得る。
%$\upsilon_0(e_i)$, $i\geq2$ が成すそれぞれ、十個の二重点及び九個の四重点すべてが、
%$\upsilon_0(e_1)$ なる三重点及び五重点の無限近傍点とはならない。
%つまり $\{\upsilon_0(e_i)\}_{i}$ は集合的にみて二点以上となる。
%そして Lemma~\ref{Lem:BoneCSzero} は、この系として得られる。

\begin{lemma}\label{Lem:BoneCSzero}
Let $f:X\to\mathbb{P}^1$ be a fibred rational surface which has a birational morphism 
$\upsilon_0:X\to\mathbb{P}^2$ giving 
$(\mathrm{B}1)$ or $(\mathrm{C})$ in {\upshape Theorem~\ref{MT}}. 
Then $\Sigma_0$ appears as the surface of a \#-minimal model of $(X, F)$. 
\end{lemma}
\begin{proof}
Let $f:X\to\mathbb{P}^1$ be a fibred rational surface of genus two. 
We may assume that there exists a birational morphism $\upsilon_0:X\to\mathbb{P}^2$ 
giving $(\mathrm{B}1)$ in {\upshape Theorem~\ref{MT}}, 
since the other case is similar and simpler. 
Let $p_{i-1}$ be the point which corresponds to $e_i$ 
as in {\upshape Theorem~\ref{MT}} by contracting it. 
For simplicity, we assume that $i<j$ if $p_j$ is an infinitely near point of $p_i$. 
Furthermore, we can take the composite $\upsilon_{i}:X\to X_i$ of the blow-downs 
which contracts $e_{11}, e_{10}, \ldots, e_{i+1}$ with $\upsilon_i(e_{i+1})=p_i$. 
If $\upsilon_1(e_i)$ is not on the minimal section $(\upsilon_1)_*e_1$ of $X_1$ for some $i\geq2$, 
then, in a way similar to the proof of the last statement in Lemma~\ref{NT}, 
$\Sigma_0$ appears as the surface of a \#-minimal model of $(X, F)$ 
by performing the elemental transformation of $X_1$ at $\upsilon_1(e_i)$. 
Assume that $p_{i+1}$ is an infinitely near point of $p_i$ for all $i$. 
Then $(e_1-e_2)$ is a $(-2)$-section of $f$, 
and $(e_i-e_{i+1})$, $i=2, 3, \ldots, 10$ are $(-2)$-curves contained in a fibre 
since they are connected.

At first, we suppose that $p_0$, $p_1$ and $p_2$ are colinear. 
Then the reducible fibre also contains the $(-2)$-curve $(\ell-e_1-e_2-e_3)$, 
since $(\ell-e_1-e_2-e_3).(e_3-e_4)=1$. 
Furthermore, only the irreducible components $(\ell-e_1-e_2-e_3)$ and $(e_i-e_{i+1})$, 
$i=2, 3, \ldots, 10$ do not generate $F$ by considering the class in 
$\mathrm{NS}(X)=\mathbb{Z}\ell\oplus(\bigoplus_{i=1}^{11}\mathbb{Z}e_i)$. 
In particular, the number of irreducible components of the reducible fibre is at least eleven. 
Now, for Jacobian surface of the generic fibre of $f$, we regard $(e_1-e_2)$ as an origin, 
and compute the rank of the Mordell-Weil group from the formula \cite[(3)~in~Theorem~3]{Shioda99}. 
In fact, a reducible fibre of $f$ is unique and 
the number of the irreducible components is exactly eleven. 
On the other hand, the unique reducible fibre is of type $(\mathrm{II}_1)$ or $(\mathrm{IV}_1)$ 
in the sense of Horikawa \cite{Hori77} from $(\ref{THindex})$. 
However, any singular fibre of type $(\mathrm{II}_1)$ or $(\mathrm{IV}_1)$ 
whose dual graph contains Dynkin diagram of the root lattice $D_{10}$ of rank ten
has more than eleven irreducible components, which is a contradiction.

Next, we assume that $p_0$, $p_1$ and $p_2$ are not colinear. 
We consider $(e_2-e_3)$, $(e_3-e_4)$ and a $(-1)$-bisection $(\ell-e_1-e_2)$ of $f$. 
Remark that they are disjoint with $(e_5-e_6), (e_6-e_7), \ldots, (e_{10}-e_{11})$ and $e_{11}$. 
Let $\upsilon_3':X\to X_3'$ be the composite of $\upsilon_{4}:X\to X_{4}$ and 
the blow-down $X_{4}\to X_{3}'$ contracting $(\upsilon_{4})_*(\ell-e_1-e_2)$. 
Then $(\upsilon_{3}')_*(e_2-e_3).(\upsilon_{3}')_*F=2$ and 
$(\upsilon_{3}')_*(e_2-e_3)$ is a $(-1)$-curve. 
We take the composite $\upsilon_{2}':X\to X_{2}'$ of $\upsilon_{3}'$ and 
the blow-down $X_{3}'\to X_2'$ contracting $(\upsilon_{3}')_*(e_2-e_3)$. 
Similarly, we can define a birational morphism $\upsilon_1':X\to X_1'$ 
as the composite of $\upsilon_{2}'$ and the blow-down $X_{2}'\to X_1'$ 
contracting $(\upsilon_{2}')_*(e_3-e_4)$. 
We regard $X_1'$ as $\Sigma_0$ so that $(\upsilon_1')_*(e_1-e_2)$ is a minimal section 
of $X_1'$ and $(\upsilon_1')_*(e_4-e_5)$ is a fibre. 
Then $(X_1', (\upsilon_{1}')_*F)$ is a \#-minimal model of $(X, F)$. 
\end{proof}

The statements in Theorem~\ref{PlanePlan} for $(\mathrm{B}1)$ and $(\mathrm{C})$ 
follow from Lemmas~\ref{Lem:BoneCStwo} and \ref{Lem:BoneCSzero}. 
When $\upsilon_0:X\to\mathbb{P}^2$ gives $(\mathrm{A})$ and $(\mathrm{B}2)$ respectively, 
$|F+K_X|$ and $|F+2K_X|$ induces an elliptic fibration with a section as follows:

\begin{theorem}[cf.\ {\upshape \cite{Kitagawa6}}]\label{Thm:A}
Let $f:X\to\mathbb{P}^1$ be a fibred rational surface which has 
a birational morphism $\upsilon_0:X\to\mathbb{P}^2$ giving $(\mathrm{A})$ 
as in {\upshape Theorem~\ref{MT}} and $(Y, G)$ the reduction of $(X, F)$. 
Then $|-K_Y|$ is a pencil of elliptic curves with one base point, and $G\sim-2K_Y$. 
Conversely, any minimal elliptic rational surface with a section induces 
a fibred rational surface of genus two as the above. 
\end{theorem}

\begin{theorem}[cf.\ {\upshape \cite{Kitagawa6}}]\label{Thm:Btwo}
Let $f:X\to\mathbb{P}^1$ be a fibred rational surface which has a birational morphism 
$\upsilon_0:X\to\mathbb{P}^2$ giving $(\mathrm{B}2)$ as in {\upshape Theorem~\ref{MT}}. 
Then the sum $(e_{10}+e_9)$ as in {\upshape Theorem~\ref{MT}} is unique. 
In particular, it is independent of a choice of $\upsilon_0$. 
Furthermore, $e_{10}$ and $e_9$ are disjoint from the $(-1)$-section $e_{11}$ of $f$. 
Let $S$ be the surface obtained from $X$ by contracting $e_{10}$ and $e_9$. 
Then $\Phi_{|-K_S|}:S\to\mathbb{P}^1$ is a minimal elliptic rational surface and 
the image of $e_{11}$ on $S$ is a section of $\Phi_{|-K_S|}$. 
Conversely, any minimal elliptic rational surface with a section induces 
a fibred rational surface of genus two as the above. 
\end{theorem}

Let $\epsilon:S\to\mathbb{P}^1$ be any minimal elliptic surface with a section $(O)_\epsilon$. 
From Theorem~$\ref{Thm:Btwo}$, we can take a subpencil of $|-3K_S+2(O)_\epsilon|$ 
whose general members have just two double points as its singularities, 
which are exactly the base points. 
Furthermore, a fibred rational surface $f:X\to\mathbb{P}^1$ which has a birational morphism 
$\upsilon_0:X\to\mathbb{P}^2$ giving $(\mathrm{B}2)$ as in {\upshape Theorem~\ref{MT}} 
is obtained from the subpencil by blowing $S$ up at the two base points. 
Then the strict transform of $(O)_\epsilon$ to $X$ is the $(-1)$-section $e_{11}$ of $f$. 
Let $(Z, H)$ be the image of $(X, F)$ by contracting $e_{11}$, $e_{10}$ and $e_9$. 
Then any \#-minimal model of the reduction $(Y, G)$ is obtained from $(Z, H)$ by contracting 
step by step a $(-1)$-curve whose intersection number with $H$ is three. 
In particular, we remark that $Z$ is obtained from $S$ by contracting $e_{11}$.

\begin{corollary}\label{Cor:nonMSzeroStwo}
For each of $d=0$ and $2$, when a birational morphism $\upsilon_0:X\to\mathbb{P}^2$ 
as in {\upshape Theorem~\ref{MT}} gives $(\mathrm{A})$ or $(\mathrm{B}2)$, 
there exists a fibred rational surface $f:X\to\mathbb{P}^1$ not admitting $\Sigma_d$ 
as the surface of a \#-minimal model of the reduction. 
\end{corollary}
\begin{proof}
Keep the situation and the notation as before Corollary~\ref{Cor:nonMSzeroStwo}. 
At first, assume that $\epsilon:S\to\mathbb{P}^1$ is a general one, 
or any fibre of $\epsilon$ is irreducible. 
Then $C^2\geq-1$ for all of irreducible curves $C$ on $S$ from the genus formula. 
Hence, there exists no birational morphism from $S$ to $\Sigma_d$ for $d\geq2$. 
Thus $\Sigma_2$ does not appear as the surface of any \#-minimal model of $(Y, G)$.

Next, we assume that $\epsilon$ has a singular fibre of type $\mathrm{II}^*$ 
in the sense of Kodaira \cite{Kodai63}. 
Then a ruling on $S$ is unique, and we can only take $\Sigma_1$ or $\Sigma_2$ 
as its relatively minimal model. 
In particular, there exists no birational morphism from $S$ to $\Sigma_0$. 
Therefore, $f:X\to\mathbb{P}^1$ does not admit $\Sigma_0$ as such a model.

The case of $(\mathrm{A})$ is similar and simpler by applying Theorem~$\ref{Thm:A}$. 
\end{proof}

\smallskip

This completes the proof of Theorem~\ref{PlanePlan}. 
\hspace*{\fill}$\square$

%
%
%
%
%
%------------------- existence -------------------
%
%
%
%
%
%\section{Existence of a section and trivial Mordell-Weil groups}
\section{Trivial Mordell-Weil groups}
Tsen's theorem saw that any ruled surface has a section. 
By mimicking the proof in \cite{Bea96}, we have the following. 
\begin{proposition}\label{section}
Every genus two fibration on a smooth projective surface whose geometric genus is zero has a section. 
\end{proposition}
\begin{proof}
Let $X$ be a smooth projective surface whose geometric genus is zero, 
$B$ a smooth projective curve and 
$f:X\to B$ a relatively minimal fibration whose general fibre $F$ is a genus two curve. 
%The N\'eron-Severi group of $X$ coincides with $H^2(X,\mathbb{Z})$, 
$\mathrm{NS}(X)$ coincides with $H^2(X,\mathbb{Z})$, 
since the geometric genus is zero. 
As $D$ runs through $H^2(X,\mathbb{Z})$, 
the set of integers $D.F$ is an ideal in $\mathbb{Z}$, 
of the form $n\mathbb{Z}$ with $n\geq1$. 
Consider the map $D \mapsto (1/n)(D.F)$ is a linear form on $H^2(X,\mathbb{Z})$. 
From Poincar\'e duality, there exists a divisor $F'\in H^2(X,\mathbb{Z})$ such that 
$D.F'=(1/n)(D.F)$ for all $D\in H^2(X,\mathbb{Z})$, 
so that $F$ is numerically equivalent to $nF'$. 
Remark that ${F'}^2=0$ from $n^2{F'}^2=F^2=0$. 
Furthermore, from the genus formula, we have $nF'.K_X=F.K_X+F^2=2$ and 
$F'.K_X=F'.(K_X+F')$ is even. 
These imply $n=1$. 
Hence, there exists a divisor $D_0\in H^2(X,\mathbb{Z})$ such that $D_0.F=1$. 
Take a sufficiently ample divisor $L$ on $B$, 
then we have an effective divisor $E\in|D_0+f^*L|$. 
Thus, there exists a component $C$ of $E$ such that $C.F=1$. 
\end{proof}

A relatively minimal fibration $f:X\to\mathbb{P}^1$ of genus two 
on a smooth rational surface always has a section from Proposition~\ref{section}. 
Furthermore, we are interested in Mordell-Weil group and lattice of $f$ 
introduced by Shioda \cite{Shioda99}.

\begin{theorem}\label{Thm:Ex}
For each $(\mathrm{A})$, $(\mathrm{B}1)$, $(\mathrm{B}2)$ or $(\mathrm{C})$ 
in {\upshape Theorem~\ref{MT}}, 
there exists a genus two fibration on a rational surface whose Mordell-Weil group is trivial. 
%{\upshape Mordell-Weil} group is trivial. 
\end{theorem}

As a proof of Theorem~\ref{Thm:Ex}, we give an example for each 
$(\mathrm{A})$, $(\mathrm{B}1)$, $(\mathrm{B}2)$ or $(\mathrm{C})$ in Theorem~\ref{MT} 
with a concrete description of $(-1)$-curves contracted by a birational morphism $\upsilon_0:X\to\mathbb{P}^2$. 
We also use the notation and the setup as before Example~\ref{Ex:Bonemts}
throughout the following four examples. 
In addition, we denote by $F_p$ the fibre of $p\in\mathbb{P}^1$ by $f$.

At first, we describe an example of $(\mathrm{A})$ in Theorem~\ref{MT} 
with trivial Mordell-Weil group.

\begin{example}\label{Ex:ATrMWG}
Let $A$ be Zariski closure on $W^\natural$ of the divisor defined by $x^5+t^3+t^2x=0$, 
which is irreducible. 
The singularities of $A$ are on $(0,0)$ and on $(\infty,\infty)$. 
Put $B^\natural=A+\varDelta_0+\varGamma_0$, which is linearly equivalent to $6\varDelta_0+4\varGamma$. 
In fact, the singularities of $B^\natural$ are 
of type $(\mathrm{V})$ on $\varGamma_0$ and of type $(0)$ on $\varGamma_\infty$. 
%In particular, $B^\natural$ is reduced. 
Take the finite double cover of $W^\natural$ branched along $B^\natural$ 
and the canonical resolution. 
Then $( \sigma^\natural )^*\varDelta_0.(-K_{W^\flat}-3( \sigma^\natural )^*\varGamma)=-1$ 
%Then $( \sigma^\natural )^*\Delta_0.(-K_{W^\flat}+{K_X}^2( \sigma^\natural )^*\Gamma)=-1$ 
implies $(\ref{eq:Rationality})$. 
Thus we have a relatively minimal fibration $f:X\to\mathbb{P}^1$ of genus two 
on a rational surface with $(K_X+F)^2=1$.

The strict transform by $\sigma\circ\widetilde{\pi}:\widetilde{X}\to W^\natural$ of $\varDelta_0$ 
is the pull-back by $\varphi:\widetilde{X}\to X$ of a double $(-1)$-section of $f$. 
Let $(O)$ be the $(-1)$-section of $f$. 
Reducible fibres of $f$ are $F_0$ and $F_\infty$. 
The irreducible components $\Theta_0, \Theta_1, \ldots, \Theta_{12}$ satisfy the following: 
Firstly, the irreducible decompositions are 
\begin{align*}
F_0&=\Theta_{11}+\Theta_9+2\Theta_{10}+2\Theta_{12}\\
F_\infty&=\Theta_0+4\Theta_1+7\Theta_2+5\Theta_8+2\sum_{i=3}^7(8-i)\Theta_i. 
\end{align*}
Secondary, $\Theta_0$ is a $(-4)$-curve and $\Theta_{12}$ is a $(-1)$-{\it elliptic curve}, 
i.e.,\ an elliptic curve whose self-intersection number is $(-1)$. 
The others are $(-2)$-curves. 
Furthermore, $\Theta_{11}$ and $\Theta_0$ intersect with $(O)$. 
Finally, the dual graphs are as in Figure~\ref{fig:H1V3}. 
%%%%%%%%%%%%%%%%%%%%%%%%%%%%%%%%%%%%%%%%%%%%%%%%%%%
%%%%%%%%%%%%%%%%%%%%%%%%%%%%%%%%%%%%%%%%%%%%%%%%%%%
%%%%%%%%%%%%%%%%%%%%%%%%%%%%%%%%%%%%%%%%%%%%%%%%%%%
\begin{figure}[hbtp]
\begin{center}
\begin{picture}(255,38.8)
\setlength{\unitlength}{1.08358pt}
%%%%%%%%%%%%%%
%%%%%%%%%%%%%%%%%
%%%     V     %%%
%%%%%%%%%%%%%%%%%
\put(-2,-10){\V}
\put(43,5){\hidariashiEllOne}
\put(50,-7){\makebox(0,0){$\scriptstyle 12$}}
%
%-------------numbering-----------------------------
%
\put(15,27.8){\makebox(0,0){$\scriptstyle 9$}}
\put(3.5,-7){\makebox(0,0){$\scriptstyle 11$}}
\put(26,-7){\makebox(0,0){$\scriptstyle 10$}}
%%%%%%%%%%%%%%%%%
%%%     V     %%%
%%%%%%%%%%%%%%%%%
\put(83,-10){\EightM}
%
%-------------numbering-----------------------------
%
\put(167.5,27.8){\makebox(0,0){$\scriptstyle 8$}}
\put(89,-7){\makebox(0,0){$\scriptstyle 0$}}
\put(111.5,-7){\makebox(0,0){$\scriptstyle 1$}}
\put(134,-7){\makebox(0,0){$\scriptstyle 2$}}
\put(156.5,-7){\makebox(0,0){$\scriptstyle 3$}}
\put(179,-7){\makebox(0,0){$\scriptstyle 4$}}
\put(201.5,-7){\makebox(0,0){$\scriptstyle 5$}}
\put(224,-7){\makebox(0,0){$\scriptstyle 6$}}
\put(246.5,-7){\makebox(0,0){$\scriptstyle 7$}}
%
%-------------numbering-----------------------------
%
\end{picture}
\caption{}\label{fig:H1V3}
\end{center}
\end{figure}
%%%%%%%%%%%%%%%%%%%%%%%%%%%%%%%%%%%%%%%%%%%%%%%%%%%
%%%%%%%%%%%%%%%%%%%%%%%%%%%%%%%%%%%%%%%%%%%%%%%%%%%
%%%%%%%%%%%%%%%%%%%%%%%%%%%%%%%%%%%%%%%%%%%%%%%%%%%
Here the numbers of the vertices corresponds to the suffixes 
%$i$'s 
of irreducible components. 
%$\Theta_i$'s. 
In particular, a double circle means that the corresponding component is an elliptic curve. 
The strict transform by $\sigma\circ\widetilde{\pi}:\widetilde{X}\to W^\natural$ of $\varDelta_\infty$ 
is the pull-back by $\varphi:\widetilde{X}\to X$ of a $(-1)$-curve. 
We denote by $e_8$ the $(-1)$-curve. 
Remark that $e_8.\Theta_{12}=e_8.\Theta_7=1$ and $e_8.F=2$.

Let $\upsilon_{11}:X\to X_{11}$ be the blow-down contracting $(O)$. 
Then $(\upsilon_{11})_*\Theta_{11}.(\upsilon_{11})_*F=1$ holds and 
$(\upsilon_{11})_*\Theta_{11}$ is a $(-1)$-curve. 
We take the composite $\upsilon_{10}:X\to X_{10}$ of $\upsilon_{11}$ and 
the blow-down $X_{11}\to X_{10}$ contracting $(\upsilon_{11})_*\Theta_{11}$. 
In this way, we have a birational morphism $\upsilon_8:X\to X_8$ 
contracting $(O)$, $\Theta_{11}$, $\Theta_{10}$ and $\Theta_9$, 
where $(X_8,(\upsilon_8)_*F)$ consists with the reduction of $(X,F)$. 
Remark that $(O)$, $\Theta_{11}$, $\Theta_{10}$ and $\Theta_9$ do not meet $e_8$. 
Let $\upsilon_7:X\to X_7$ be the composite of $\upsilon_8$ and 
the blow-down $X_{8}\to X_{7}$ contracting $(\upsilon_8)_*e_8$. 
Then $(\upsilon_7)_*\Theta_7.(\upsilon_7)_*F=2$ and $(\upsilon_7)_*\Theta_7$ is a $(-1)$-curve. 
Similarly, we define a birational morphism $\upsilon_i:X\to X_i$ 
as the composite of $\upsilon_{i+1}$ and the blow-down $X_{i+1}\to X_{i}$ 
contracting $(\upsilon_{i+1})_*\Theta_{i+1}$ for $i=6,5,\ldots,0$. 
Then $\upsilon_0:X\to X_0=\mathbb{P}^2$ is a birational morphism giving $(\mathrm{A})$ in Theorem~\ref{MT}. 
Furthermore, we have 
\begin{align*}
(O)=e_{12}, \ \ \ 
\Theta_0 &=\ell -e_1-e_9-e_{10}-e_{11}-e_{12}, \ \ \ 
\Theta_8 =\ell -e_1-e_2-e_3, 
\\
\Theta_{12}&=3\ell -e_1-e_2-\cdots -e_{10}, \ \ \ 
\Theta_i = e_i-e_{i+1}, \ \ \ i=1,2,\ldots, 7, 9, 10, 11,
\end{align*}
and $\mathrm{NS}(X)=\mathbb{Z}\ell\oplus(\bigoplus_{i=1}^{12}\mathbb{Z}e_i)$, 
where $\ell$ and $e_i$ denote the same in Theorem~\ref{MT}. 
In addition, an orthogonal decomposition of the N\'eron-Severi lattice $\mathfrak{NS}(X)$, 
that is, $\mathrm{NS}(X)$ equipped with the bilinear form 
which is $(-1)$ times of the intersection form, 
is as follows: 
%\begin{align}
\begin{align*}
%\label{eq:ODA}
\mathfrak{NS}(X)=(\mathbb{Z}F\oplus\mathbb{Z}(O))\oplus^\perp
(\mathbb{Z}\Theta_{9}\oplus\mathbb{Z}\Theta_{10}\oplus\mathbb{Z}\Theta_{12})\oplus^\perp
\left(\bigoplus_{i=1}^{8}\mathbb{Z}\Theta_i\right). 
\end{align*}
%\end{align}
Here $\mathfrak{L} \oplus^\perp \mathfrak{M} \oplus^\perp \cdots \oplus^\perp \mathfrak{N}$ 
means that lattices $\mathfrak{L},\mathfrak{M},\ldots,\mathfrak{N}$ are orthogonal each other. 
Thus, Mordell-Weil group of $f$ is trivial from \cite[Theorem~1]{Shioda99}. 
\end{example}

Secondly, we take an example of $(\mathrm{B}1)$ in Theorem~\ref{MT}
with trivial Mordell-Weil group.

\begin{example}\label{Ex:BoneTrMWG}
Let $A$ be Zariski closure on $W^\natural$ of the divisor defined by 
%$x^5+tx^4+tx^3+t^2x+t^3=0$, which is irreducible. 
\begin{align*}
x^5+tx^4+tx^3+t^2x+t^3=0, 
\end{align*}
which is irreducible. 
The singularities of $A$ are on $(0,0)$ and on $(\infty,\infty)$. 
Take the section $C$ of $pr_1$ defined by $x+t=0$. 
Both simple triple points $(0,0)$ and $(\infty,\infty)$ of $A$ 
have contacts of order four with $C$. 
Put $B^\natural=A+C+\varGamma_0+\varGamma_\infty$, 
which is linearly equivalent to $6\varDelta_0+6\varGamma$. 
In fact, the singularities of $B^\natural$ are 
of type $(\mathrm{V})$ on $\varGamma_0$ and of same type on $\varGamma_\infty$. 
%In particular, $B^\natural$ is reduced. 
We take the finite double cover of $W^\natural$ branched along $B^\natural$ 
and the canonical resolution. 
Then we also have $(\ref{eq:Rationality})$ 
since $B^\natural$ does not include any minimal section of $pr_1$. 
Therefore, we obtain a relatively minimal fibration $f:X\to\mathbb{P}^1$ of genus two 
on a rational surface with $(K_X+F)^2=2$.

The strict transform by $\sigma\circ\widetilde{\pi}:\widetilde{X}\to W^\natural$ of 
$C$ is the pull-back by $\varphi:\widetilde{X}\to X$ of a double $(-2)$-section of $f$. 
Let $(O)$ be the $(-2)$-section of $f$. 
Reducible fibres of $f$ are $F_0$ and $F_\infty$. 
Furthermore, the irreducible components $\Theta_0, \Theta_1, \ldots, \Theta_{11}$ 
satisfy the following: 
Firstly, the irreducible decompositions are 
\begin{align*}
F_0=\Theta_0 +\Theta_7 +2\sum_{i=8}^{11}\Theta_i, \ \ \
F_\infty =\Theta_1+\Theta_2+2\sum_{i=3}^6\Theta_i. 
\end{align*}
Secondary, $\Theta_6$ and $\Theta_{11}$ are $(-1)$-elliptic curves. 
The others are $(-2)$-curves. 
Furthermore, $\Theta_0$ and $\Theta_2$ intersect with $(O)$. 
Finally, the dual graphs are as in Figure~\ref{fig:H2VV}. 
%%%%%%%%%%%%%%%%%%%%%%%%%%%%%%%%%%%%%%%%%%%%%%%%%%%
%%%%%%%%%%%%%%%%%%%%%%%%%%%%%%%%%%%%%%%%%%%%%%%%%%%
%%%%%%%%%%%%%%%%%%%%%%%%%%%%%%%%%%%%%%%%%%%%%%%%%%%
\begin{figure}[hbtp]
\begin{center}
\begin{picture}(258,38.8)
\setlength{\unitlength}{1.08358pt}
%%%%%%%%%%%%%%%%%
%%%     V     %%%
%%%%%%%%%%%%%%%%%
\put(-2,-9){\V}
\put(43,6){\migite}
\put(65.5,6){\migite}
\put(87.8,6){\hidariashiEllOne}
%
%-------------numbering-----------------------------
%
\put(15,28.8){\makebox(0,0){$\scriptstyle 7$}}
\put(4,-6){\makebox(0,0){$\scriptstyle 0$}}
\put(26.5,-6){\makebox(0,0){$\scriptstyle 8$}}
\put(49,-6){\makebox(0,0){$\scriptstyle 9$}}
\put(71.5,-6){\makebox(0,0){$\scriptstyle 10$}}
\put(95.5,-6){\makebox(0,0){$\scriptstyle 11$}}
%
%-------------numbering-----------------------------
%
%%%%%%%%%%%%%%%%%
%%%     V     %%%
%%%%%%%%%%%%%%%%%
\put(133,-9){\V}
\put(178,6){\migite}
\put(200.5,6){\migite}
\put(222.8,6){\hidariashiEllOne}
%
%-------------numbering-----------------------------
%
\put(150,28.8){\makebox(0,0){$\scriptstyle 1$}}
\put(139,-6){\makebox(0,0){$\scriptstyle 2$}}
\put(161.5,-6){\makebox(0,0){$\scriptstyle 3$}}
\put(184,-6){\makebox(0,0){$\scriptstyle 4$}}
\put(206.5,-6){\makebox(0,0){$\scriptstyle 5$}}
\put(230.5,-6){\makebox(0,0){$\scriptstyle 6$}}
%
%-------------numbering-----------------------------
%
\end{picture}
\caption{}\label{fig:H2VV}
\end{center}
\end{figure}
%%%%%%%%%%%%%%%%%%%%%%%%%%%%%%%%%%%%%%%%%%%%%%%%%%%
%%%%%%%%%%%%%%%%%%%%%%%%%%%%%%%%%%%%%%%%%%%%%%%%%%%
%%%%%%%%%%%%%%%%%%%%%%%%%%%%%%%%%%%%%%%%%%%%%%%%%%%
Here the numbers of the vertices corresponds to the suffixes 
of irreducible components. 
The strict transform by $\sigma\circ\widetilde{\pi}:\widetilde{X}\to W^\natural$ of $\varDelta_0$ 
is the pull-back by $\varphi:\widetilde{X}\to X$ of a $(-1)$-curve. 
We denote by $e_{11}$ the $(-1)$-curve. 
That of $\varDelta_\infty$ is the pull-back by $\varphi:\widetilde{X}\to X$ of a $(-1)$-curve. 
Let $e_6$ be the $(-1)$-curve. 
Remark that $e_6.\Theta_{5}=e_{11}.\Theta_{10}=1$ and $e_6.F=e_{11}.F=2$.

Let $\upsilon_{10}:X\to X_{10}$ be the blow-down contracting $e_{11}$. 
Then $(\upsilon_{10})_*\Theta_{10}.(\upsilon_{10})_*F=2$ and 
$(\upsilon_{10})_*\Theta_{10}$ is a $(-1)$-curve. 
We take the composite $\upsilon_{9}:X\to X_{9}$ of $\upsilon_{10}$ and 
the blow-down $X_{10}\to X_9$ contracting $(\upsilon_{10})_*\Theta_{10}$. 
Similarly, we define a birational morphism $\upsilon_i:X\to X_i$ 
as the composite of $\upsilon_{i+1}$ and the blow-down $X_{i+1}\to X_i$ 
contracting $(\upsilon_{i+1})_*\Theta_{i+1}$ for $i=8,7,6$. 
Remark that $e_{11}$, $\Theta_{10}$, $\Theta_{9}$, $\Theta_{8}$ and $\Theta_{7}$ 
do not meet $e_6$. 
Let $\upsilon_{5}:X\to X_{5}$ be the composite of $\upsilon_{6}$ and 
the blow-down $X_{6}\to X_5$ contracting $(\upsilon_{6})_*e_{6}$. 
Furthermore, we also define a birational morphism $\upsilon_i:X\to X_i$ 
as the composite of $\upsilon_{i+1}$ and the blow-down $X_{i+1}\to X_i$ 
contracting $(\upsilon_{i+1})_*\Theta_{i+1}$ for $i=4,3,2,1$. 
Then $(\upsilon_1)_*(O).(\upsilon_1)_*F=3$ and $(\upsilon_1)_*(O)$ is a $(-1)$-curve. 
Hence, the composite $\upsilon_0:X\to\mathbb{P}^2$ of $\upsilon_{1}$ and 
the blow-down $X_1\to\mathbb{P}^2$ contracting $(\upsilon_{1})_*(O)$ 
is a birational morphism giving $(\mathrm{B}1)$ in Theorem~\ref{MT}. 
Furthermore, we have 
\begin{align*}
\Theta_0 &=\ell -e_1-e_7-e_8, \ \ \ \Theta_1 =\ell -e_1-e_2-e_3, \ \ \ 
\Theta_6 =3\ell -\sum_{i=1}^5e_i-\sum_{i=7}^{11}e_i, \\
\Theta_i &= e_i-e_{i+1}, \ \ \ i=2, 3, 4, 5, 7, 8, 9, 10, \ \ \ 
\Theta_{11} =3\ell -\sum_{i=1}^{10}e_i, \ \ \ 
(O)=e_1-e_2,
\end{align*}
where $\ell$ and $e_i$ denote the same in Theorem~\ref{MT}. 
In addition, we have an orthogonal decomposition 
%\begin{align}
\begin{align*}
%\label{eq:ODB1}
\mathfrak{NS}(X)=(\mathbb{Z}F\oplus\mathbb{Z}(O))\oplus^\perp
\left(\bigoplus_{i=7}^{11}\mathbb{Z}\Theta_i\right)\oplus^\perp
\left(\mathbb{Z}\Theta_1\oplus\bigoplus_{i=3}^{6}\mathbb{Z}\Theta_i\right). 
%\end{align}
\end{align*}
Thus, Mordell-Weil group of $f$ is trivial from \cite[Theorem~1]{Shioda99}. 
\end{example}

Next, we show an example of $(\mathrm{B}2)$ in Theorem~\ref{MT} 
with trivial Mordell-Weil group.

\begin{example}\label{Ex:BtwoTrMWG}
Let $A$ be Zariski closure on $W^\natural$ of the divisor defined by 
\begin{align*}
t^3x^5-3t^4x^4-3t^2x^4+6t^3x^3+3tx^3+3t^4x^2-3t^2x^2-x^2-3t^3x-t^4=0, 
\end{align*}
which is irreducible. 
The singularities of $A$ are on $(0,0)$ and on $(0,\infty)$. 
Let $C$ be the section of $pr_1$ defined by $tx-1=0$. 
The three-fold triple point $(0,\infty)$ of $A$ has 
a contact of order nine with $C$. 
Put $B^\natural=A+\varDelta_0$, which is linearly equivalent to $6\varDelta_0+4\varGamma$. 
In fact, the singularities of $B^\natural$ are of type $(\mathrm{II}_1)$ on $\varGamma_0$. 
%In particular, $B^\natural$ is reduced. 
We take the finite double cover of $W^\natural$ branched along $B^\natural$ 
and the canonical resolution. 
Since $B^\natural$ includes just one minimal section $\varDelta_0$ of $pr_1$, 
we also have $(\ref{eq:Rationality})$. 
Thus we obtain a relatively minimal fibration $f:X\to\mathbb{P}^1$ of genus two 
on a rational surface with $(K_X+F)^2=2$.

The strict transform by $\sigma\circ\widetilde{\pi}:\widetilde{X}\to W^\natural$ of $\varDelta_0$ 
is the pull-back by $\varphi:\widetilde{X}\to X$ of a double $(-1)$-section of $f$. 
Let $(O)$ be the $(-1)$-section of $f$. 
Now, $F_0$ is a unique reducible fibre of $f$, and 
the irreducible components $\Theta_0, \Theta_1, \ldots, \Theta_{10}$ satisfy the following: 
Firstly, the irreducible decomposition is 
\begin{align*}
F_0=\Theta_{10}+\Theta_9 +\sum_{i=3}^8(9-i)\Theta_i+4\Theta_2+2\Theta_1+3\Theta_0. 
\end{align*}
Secondary, $\Theta_8$ is a $(-3)$-curve and $\Theta_{10}$ is a $(-1)$-elliptic curve. 
The others are $(-2)$-curves. 
Furthermore, $\Theta_{10}$ intersects with $(O)$. 
Finally, the dual graph is as in Figure~\ref{fig:H2II1}. 
%%%%%%%%%%%%%%%%%%%%%%%%%%%%%%%%%%%%%%%%%%%%%%%%%%%
%%%%%%%%%%%%%%%%%%%%%%%%%%%%%%%%%%%%%%%%%%%%%%%%%%%
%%%%%%%%%%%%%%%%%%%%%%%%%%%%%%%%%%%%%%%%%%%%%%%%%%%
\begin{figure}[hbtp]
\begin{center}
\begin{picture}(238,38.8)
\setlength{\unitlength}{1.08358pt}
\put(0,5){\ellOne}
\put(22.5,5){\migite}
\put(45,5){\migitesan}
\put(67.5,-10){\Eight}
%
%-------------numbering-----------------------------
%
\put(6.1,-7){\makebox(0,0){$\scriptstyle 10$}}
\put(28.5,-7){\makebox(0,0){$\scriptstyle 9$}}
\put(51,-7){\makebox(0,0){$\scriptstyle 8$}}
\put(73.5,-7){\makebox(0,0){$\scriptstyle 7$}}
\put(96,-7){\makebox(0,0){$\scriptstyle 6$}}
\put(118.5,-7){\makebox(0,0){$\scriptstyle 5$}}
\put(141,-7){\makebox(0,0){$\scriptstyle 4$}}
\put(163.5,-7){\makebox(0,0){$\scriptstyle 3$}}
\put(186,-7){\makebox(0,0){$\scriptstyle 2$}}
\put(208.5,-7){\makebox(0,0){$\scriptstyle 1$}}
\put(175,27.8){\makebox(0,0){$\scriptstyle 0$}}
%
%-------------numbering-----------------------------
%

\end{picture}
\caption{}\label{fig:H2II1}
\end{center}
\end{figure}
%%%%%%%%%%%%%%%%%%%%%%%%%%%%%%%%%%%%%%%%%%%%%%%%%%%
%%%%%%%%%%%%%%%%%%%%%%%%%%%%%%%%%%%%%%%%%%%%%%%%%%%
%%%%%%%%%%%%%%%%%%%%%%%%%%%%%%%%%%%%%%%%%%%%%%%%%%%
Here the numbers of the vertices corresponds to the suffixes of irreducible components. 
The strict transform by $\sigma\circ\widetilde{\pi}:\widetilde{X}\to W^\natural$ of $\varDelta_\infty$ 
is the pull-back by $\varphi:\widetilde{X}\to X$ of a $(-1)$-curve. 
Let $e_{10}$ be the $(-1)$-curve. 
That of $C$ is the pull-back by $\varphi:\widetilde{X}\to X$ of a $(-2)$-curve. 
We denote by $\widehat{e_8}$ the $(-2)$-curve. 
Remark that 
$e_{10}.\Theta_{9}=e_{10}.\Theta_{8}=\widehat{e_8}.(O)=\widehat{e_8}.\Theta_{7}=1$ and 
$e_{10}.F=\widehat{e_8}.F=2$.

Let $\upsilon_{10}:X\to X_{10}$ be the blow-down contracting $(O)$. 
Since $(O)$ do not meet $e_{10}$, 
we take the composite $\upsilon_{9}:X\to X_{9}$ of $\upsilon_{10}$ and 
the blow-down $X_{10}\to X_{9}$ contracting $(\upsilon_{10})_*e_{10}$. 
Then $(\upsilon_9)_*\Theta_9.(\upsilon_9)_*F=2$ and $(\upsilon_9)_*\Theta_9$ is a $(-1)$-curve. 
We denote by $\upsilon_8:X\to X_8$ the composite of $\upsilon_9$ and 
the blow-down $X_{9}\to X_{8}$ contracting $(\upsilon_{9})_*\Theta_{9}$. 
Remark that $(\upsilon_8)_*\widehat{e_8}.(\upsilon_8)_*F=3$ and 
$(\upsilon_8)_*\widehat{e_8}$ is a $(-1)$-curve. 
Let $\upsilon_7:X\to X_7$ be the composite of $\upsilon_8$ and 
the blow-down $X_{8}\to X_{7}$ contracting $(\upsilon_8)_*\widehat{e_8}$. 
Similarly, we define a birational morphism $\upsilon_i:X\to X_i$ 
as the composite of $\upsilon_{i+1}$ and the blow-down $X_{i+1}\to X_{i}$ 
contracting $(\upsilon_{i+1})_*\Theta_{i+1}$ for $i=6,5,\ldots,0$. 
Then $\upsilon_0:X\to X_0=\mathbb{P}^2$ is a birational morphism giving $(\mathrm{B}2)$ in Theorem~\ref{MT}. 
Furthermore, we have 
\begin{align*}
(O)=e_{11}, \ \ \ 
\Theta_{10} &=3\ell -\sum_{i=1}^9e_i-e_{11}, \ \ \ 
\Theta_0 =\ell -e_1-e_2-e_3, 
\\
\Theta_{8} &=3\ell -\sum_{i=1}^7e_i-2e_9-e_{10}, 
\ \ \ 
\Theta_i = e_i-e_{i+1}, \ \ \ i=1,2,\ldots, 6, 7, 9, 
\end{align*}
where $\ell$ and $e_i$ denote the same in Theorem~\ref{MT}. 
Hence, we have an orthogonal decomposition 
%\begin{align}
\begin{align*}
%\label{eq:ODB2}
\mathfrak{NS}(X)=(\mathbb{Z}F\oplus\mathbb{Z}(O))\oplus^\perp
\left(\bigoplus_{i=0}^{9}\mathbb{Z}\Theta_i\right). 
%\end{align}
\end{align*}
%$H^2(X,\mathbb{Z})=\mathbb{Z}F\oplus\mathbb{Z}e_{11}\oplus\bigoplus_{i=0}^{9}\mathbb{Z}\Theta_i$
Thus, Mordell-Weil group of $f$ is trivial from \cite[Theorem~1]{Shioda99}. 
\end{example}

In the last, we see an example of $(\mathrm{C})$ in Theorem~\ref{MT} 
with trivial Mordell-Weil group.

\begin{example}\label{Ex:CTrMWG}
Let $A$ be Zariski closure on $W^\natural$ of the divisor defined by 
\begin{align*}
x^5+t^2x^4-6tx^4+6t^2x^3+4tx^3-4t^3x^2-6t^2x^2+6t^3x-t^2x-t^4=0, 
\end{align*}
which is irreducible. 
The singularities of $A$ are on $(0,0)$, on $(1,1)$ and on $(\infty,\infty)$. 
The section $C$ of $pr_1$ defined by $x-t=0$ passes through 
$(0,0)$, $(1,1)$ and $(\infty,\infty)$. 
Put $B^\natural=A+C+\varGamma_0+\varGamma_1+\varGamma_\infty$, 
which is linearly equivalent to $6\varDelta_0+8\varGamma$. 
In fact, the singularities of $B^\natural$ are three of type $(\mathrm{V})$. 
They are on $\varGamma_0$, on $\varGamma_1$ and on $\varGamma_\infty$. 
%In particular, $B^\natural$ is reduced. 
We define sections $D_8$, $D_9$ and $D_{10}$ of $pr_2$ by the following:
\begin{align*}
D_8: x^2-2x+t=0, \ \ \ 
D_9: x^2-t=0, \ \ \ 
D_{10}: x^2-2tx+t=0.
\end{align*}
They are linearly equivalent to $2\varDelta_0+\varGamma$, 
and pass through $(0,0)$, $(1,1)$ and $(\infty,\infty)$. 
Furthermore, $D_8$ is tangent to $\varGamma_1$ and to $\varGamma_\infty$. 
In fact, $D_8$ is a unique curve having such properties. 
However, it meets $\varGamma_0$ transversally. 
Therefore, $(\ref{eq:Rationality})$ also holds when 
we take the finite double cover of $W^\natural$ branched along $B^\natural$ 
and the canonical resolution. 
Hence, we obtain a relatively minimal fibration $f:X\to\mathbb{P}^1$ of genus two 
on a rational surface with $(K_X+F)^2=3$. 
Indeed, $D_9$ meets $\varGamma_1$ transversally, 
though it is also tangent to $\varGamma_0$ and to $\varGamma_\infty$. 
Similarly, $D_{10}$ is tangent to $\varGamma_0$ and to $\varGamma_1$, 
though it meets $\varGamma_\infty$ transversally.

The strict transform by $\sigma\circ\widetilde{\pi}:\widetilde{X}\to W^\natural$ of 
$C$ is the pull-back by $\varphi:\widetilde{X}\to X$ of a double $(-2)$-section of $f$. 
Let $(O)$ be the $(-2)$-section of $f$. 
Reducible fibres of $f$ are $F_0$, $F_1$ and $F_\infty$. 
Furthermore, the irreducible components $\Theta_2, \Theta_3, \ldots, \Theta_{13}$ 
satisfy the following: 
Firstly, the irreducible decompositions are 
\begin{align*}
F_0&=\Theta_2 +\Theta_{11} +2\Theta_5 +2\Theta_8, \\
F_1&=\Theta_3 +\Theta_{12} +2\Theta_6 +2\Theta_9, \ \ \
F_\infty =\Theta_4 +\Theta_{13} +2\Theta_7 +2\Theta_{10}. 
\end{align*}
Secondary, $\Theta_8$, $\Theta_9$ and $\Theta_{10}$ are $(-1)$-elliptic curves. 
The others are $(-2)$-curves. 
Furthermore, $\Theta_2$, $\Theta_3$ and $\Theta_4$ intersect with $(O)$. 
Finally, the dual graphs are as in Figure~\ref{fig:H3VVV}. 
%%%%%%%%%%%%%%%%%%%%%%%%%%%%%%%%%%%%%%%%%%%%%%%%%%%
%%%%%%%%%%%%%%%%%%%%%%%%%%%%%%%%%%%%%%%%%%%%%%%%%%%
%%%%%%%%%%%%%%%%%%%%%%%%%%%%%%%%%%%%%%%%%%%%%%%%%%%
\begin{figure}[hbtp]
\begin{center}
\begin{picture}(258,38.8)
\setlength{\unitlength}{1.08358pt}
%%%%%%%%%%%%%%%%%
%%%     V     %%%
%%%%%%%%%%%%%%%%%
\put(-2,-9){\V}
\put(42.5,6){\hidariashiEllOne}
%
%-------------numbering-----------------------------
%
\put(15,28.8){\makebox(0,0){$\scriptstyle 11$}}
\put(4,-6){\makebox(0,0){$\scriptstyle 2$}}
\put(26.5,-6){\makebox(0,0){$\scriptstyle 5$}}
\put(50,-6){\makebox(0,0){$\scriptstyle 8$}}
%
%-------------numbering-----------------------------
%
%%%%%%%%%%%%%%%%%
%%%     V     %%%
%%%%%%%%%%%%%%%%%
\put(86,-9){\V}
\put(130.5,6){\hidariashiEllOne}
%
%-------------numbering-----------------------------
%
\put(103,28.8){\makebox(0,0){$\scriptstyle 12$}}
\put(92,-6){\makebox(0,0){$\scriptstyle 3$}}
\put(114.5,-6){\makebox(0,0){$\scriptstyle 6$}}
\put(138,-6){\makebox(0,0){$\scriptstyle 9$}}
%
%-------------numbering-----------------------------
%
%%%%%%%%%%%%%%%%%
%%%     V     %%%
%%%%%%%%%%%%%%%%%
\put(174,-9){\V}
\put(218.5,6){\hidariashiEllOne}
%
%-------------numbering-----------------------------
%
\put(191,28.8){\makebox(0,0){$\scriptstyle 13$}}
\put(180,-6){\makebox(0,0){$\scriptstyle 4$}}
\put(202.5,-6){\makebox(0,0){$\scriptstyle 7$}}
\put(226,-6){\makebox(0,0){$\scriptstyle 10$}}
%
%-------------numbering-----------------------------
%
\end{picture}
\caption{}\label{fig:H3VVV}
\end{center}
\end{figure}
%%%%%%%%%%%%%%%%%%%%%%%%%%%%%%%%%%%%%%%%%%%%%%%%%%%
%%%%%%%%%%%%%%%%%%%%%%%%%%%%%%%%%%%%%%%%%%%%%%%%%%%
%%%%%%%%%%%%%%%%%%%%%%%%%%%%%%%%%%%%%%%%%%%%%%%%%%%
Here the numbers of the vertices corresponds to the suffixes of irreducible components. 
For $i=8,9,10$, 
the strict transform by $\sigma\circ\widetilde{\pi}:\widetilde{X}\to W^\natural$ of $D_i$ 
is the pull-back by $\varphi:\widetilde{X}\to X$ of a $(-1)$-curve. 
We denote by $e_{i}$ the $(-1)$-curve for $i=8,9,10$. 
That of $A$ is the pull-back by $\varphi:\widetilde{X}\to X$ of a double $(-1)$-curve. 
We denote by $e_1$ the $(-1)$-curve. 
Remark that 
\begin{align*}
&e_8.\Theta_{5}=e_9.\Theta_{6}=e_{10}.\Theta_{7}=e_1.\Theta_{11}=e_1.\Theta_{12}=e_1.\Theta_{13}=1, \\
&e_1.\Theta_{8}=e_1.\Theta_{9}=e_1.\Theta_{10}=2, \ \ \ e_8.F=e_9.F=e_{10}.F=4
\end{align*}
%$e_8.F=e_9.F=e_{10}.F=4$ 
and $e_1.F=5$.

Since $e_8$, $e_9$ and $e_{10}$ are three disjoint $(-1)$-curves, 
we define a birational morphism $\upsilon_7:X\to X_7$ 
as the composite of the blow-downs which contract them. 
Then $(\upsilon_{7})_*\Theta_{7}.(\upsilon_{7})_*F=4$ and 
$(\upsilon_{7})_*\Theta_{7}$ is a $(-1)$-curve. 
We take the composite $\upsilon_{6}:X\to X_{6}$ of $\upsilon_{7}$ and 
the blow-down $X_{7}\to X_6$ contracting $(\upsilon_{7})_*\Theta_{7}$. 
Similarly, we define a birational morphism $\upsilon_i:X\to X_i$ 
as the composite of $\upsilon_{i+1}$ and the blow-down $X_{i+1}\to X_i$ 
contracting $(\upsilon_{i+1})_*\Theta_{i+1}$ for $i=5,4,\ldots,1$. 
Remark that $\Theta_{2}, \Theta_{3}, \ldots, \Theta_{7}$, $e_{8}$, $e_{9}$ and $e_{10}$ 
do not meet $e_1$. 
Hence, the composite $\upsilon_0:X\to\mathbb{P}^2$ of $\upsilon_{1}$ and 
the blow-down $X_1\to\mathbb{P}^2$ contracting $(\upsilon_{1})_*e_1$ 
is a birational morphism giving $(\mathrm{C})$ in Theorem~\ref{MT}. 
Furthermore, we have 
\begin{align*}
\Theta_i &= e_i-e_{i+3}, \ \ \ i=2, 3, \ldots, 7, \ \ \ \ \ 
\Theta_i =6\ell -2\sum_{j=1}^{10}e_j+e_i, \ \ \ i=8,9,10, \ \ \ 
\\
\Theta_{i+9} &=\ell -e_1-e_i-e_{i+3}, \ \ \ i=2,3,4, \ \ \ \ \ 
(O)=\ell -e_2-e_3-e_4,
\end{align*}
where $\ell$ and $e_i$ denote the same in Theorem~\ref{MT}. 
Therefore, we have an orthogonal decomposition 
%\begin{align}
\begin{align*}
%H^2(X,\mathbb{Z})=(\mathbb{Z}F\oplus\mathbb{Z}(O))\oplus^\perp
%\left(\bigoplus_{i=7}^{11}\mathbb{Z}\Theta_i\right)\oplus^\perp
%\left(\mathbb{Z}\Theta_1\oplus\bigoplus_{i=3}^{6}\mathbb{Z}\Theta_i\right). 
%\label{eq:ODC}
\mathfrak{NS}(X)=
(\mathbb{Z}F\oplus\mathbb{Z}(O))\oplus^\perp
&(\mathbb{Z}\Theta_5\oplus\mathbb{Z}\Theta_8\oplus\mathbb{Z}\Theta_{11})
\\
\oplus^\perp
&(\mathbb{Z}\Theta_6\oplus\mathbb{Z}\Theta_9\oplus\mathbb{Z}\Theta_{12})\oplus^\perp
(\mathbb{Z}\Theta_7\oplus\mathbb{Z}\Theta_{10}\oplus\mathbb{Z}\Theta_{13}). 
%\notag
\end{align*}
%\end{align}
Thus, Mordell-Weil group of $f$ is trivial from \cite[Theorem~1]{Shioda99}. 
\end{example}

\smallskip

This completes the proof of Theorem~\ref{Thm:Ex}. 
\hspace*{\fill}$\square$

%\begin{remark}
%Example~\ref{Ex:BoneTrMWG} で $\Sigma_2$ へのモデルの取替えも紹介する。
%
%Example~\ref{Ex:Bonemts} と比較せよ。
%
%
%
%Corollary~\ref{Cor:nonMSzeroStwo} で示した通り, 
%Examples~\ref{Ex:ATrMWG} and \ref{Ex:BtwoTrMWG} で $Y$ を経由した後は 
%$\Sigma_0$ にモデルがとれない事を紹介する。
%\end{remark}

\begin{remark}\label{ABM}
Keep the situation and the notation as in Example~\ref{Ex:CTrMWG}. 
%Return the situation and the notation as in Example~\ref{Ex:CTrMWG}. 
Remark that $(\upsilon_3)_*\Theta_{12}$, $(\upsilon_3)_*\Theta_{11}$ and $(\upsilon_3)_*(O)$ 
are three disjoint $(-1)$-curves on $X_3$. 
Hence, we obtain another birational morphism $\upsilon_0':X\to\mathbb{P}^2$ 
%Hence, we obtain another birational morphism $\upsilon_0':X\to(\mathbb{P}^2)'$ 
from $\upsilon_3:X \to X_3$ by contracting them. 
Then $\deg\upsilon_0'(F)=13$ and singularities of $\upsilon_0'(F)$ are the same as $\upsilon_0(F)$'s. 
In fact, $\upsilon_0'\circ(\upsilon_0)^{-1}: \mathbb{P}^2\dashrightarrow\mathbb{P}^2$ is the Cremona 
%In fact, $\upsilon_0'\circ(\upsilon_0)^{-1}: \mathbb{P}^2\dashrightarrow(\mathbb{P}^2)'$ is the Cremona 
transformation at three points $\upsilon_0(e_1)$, $\upsilon_0(e_2)$ and $\upsilon_0(e_3)$. 
\end{remark}

\begin{remark}\label{Rmk:SINsec}
Although the self-intersection number of any section of the genus two fibrations in 
Examples~\ref{Ex:Bonemts} and \ref{Ex:Cmts} is at most $(-3)$, 
that of the unique sections of the ones in Examples~\ref{Ex:BoneTrMWG} and \ref{Ex:CTrMWG} 
is equal to $(-2)$. 
In contrast, $f:X\to\mathbb{P}^2$ always has a $(-1)$-section 
when $\upsilon_0:X\to\mathbb{P}^2$ gives $(\mathrm{A})$ or $(\mathrm{B}2)$. 
\end{remark}

%\begin{remark}
%主定理より $\rho(X)=14$ かつ $\mathrm{MWG}(f)$ が自明な場合を除いて、相対標準写像は 
%$\mathbb{P}^2$ の十三点ブローアップから $\Sigma_0$ への 
%generially finite double cover へ lift する。
%この観点から、その例外的な場合を除いて、十把一絡に扱ったら何か良いことが起こらないか。
%\end{remark}
%
%=%=%=%=%=%=%=%=%=%=%=%=%=%=%=%=%=%=%=%=%=%=%=%=%=%=%=%=%=%=%=%=%=%=%=%=%=%=%=%=%=%=%
%
%\newpage

\bigskip
%%%%%%%%%%%% 著者所属 %%%%%%%%%%%%%
% 第一著者
\address{
General Education, \\
Gifu National College of Technology, \\
Kamimakuwa 2236-2, Motosu, \\ 
Gifu 501-0495, Japan
}
%{shinya (A)cr.math.sci.osaka-u.ac.jp}
%{shinya\\ \hspace*{\fill}@cr.math.sci.osaka-u.ac.jp}
{kit058shiny@gifu-nct.ac.jp}
%{shinya@cr.math.\\ \hspace{80pt}sci.osaka-u.ac.jp}
\end{document}